\documentclass[12pt, a4paper, reqno]{amsart}
\DeclareMathAlphabet{\mathscrbf}{OMS}{mdugm}{b}{n}

\usepackage[utf8]{inputenc}
\usepackage[slovene, english]{babel}
\usepackage{polski}
\usepackage{amsmath,amsfonts,amssymb,mathrsfs,amsthm,amscd, ulem, stmaryrd }
\usepackage[pdftex]{graphicx}
\usepackage{tikz-cd}
\usepackage{soul} 
\usepackage{MnSymbol}
\usepackage[all,cmtip]{xy}
\usepackage{enumitem}
\usepackage{manfnt}

\usepackage[gen]{eurosym}

\usepackage{slashed}

\usepackage{pgfplots}

\usepackage[mathcal]{euscript}  % the nice mathcal font

\usepackage{url}

\usetikzlibrary{arrows}

 \usepackage{relsize}

\tikzset{
  no line/.style={draw=none,
    commutative diagrams/every label/.append style={/tikz/auto=false}},
  from/.style args={#1 to #2}{to path={(#1)--(#2)\tikztonodes}}}

\usetikzlibrary{circuits.ee.IEC}

\title{Prismatization via spherical loop spaces}
\author{Rok Gregoric}
\thanks{Johns Hopkins University}
\date{\today}
\address{Johns Hopkins University, Baltimore, MD 21218, USA}
\email{rgrego12@jhu.edu}

\newtheorem{theorem}{Theorem}[section]

\newtheorem*{theoremm}{Main Theorem}
\newtheorem{theoremmm}{Theorem}

\newtheorem{lemma}[theorem]{Lemma}
\newtheorem{prop}[theorem]{Proposition}

\theoremstyle{definition}

\newtheorem{cons}[theorem]{Construction}

\newtheorem{summary}[theorem]{Summary}

\newtheorem{remark}[theorem]{Remark}

\usepackage{tikz, calc}
\usetikzlibrary{matrix,arrows}

\newcommand*{\CAlg}{{\operatorname{CAlg}}}

\newcommand*{\mM}{\mathcal M}

\newcommand*{\mX}{\mathcal X}

\newcommand*{\sO}{\mathcal O}

\newcommand*{\E}{\mathbb E_\infty}

\newcommand*{\heart}{\heartsuit}

\newcommand*{\sheafhom}{\mathscr{H}\kern -.5pt om}
\DeclareMathOperator{\Novak}{\mathscr{N}\text{\kern -3pt {\calligra\large ovak}}\,\,}

\usepackage{amsmath,calligra,mathrsfs}
\DeclareMathOperator{\fHom}{\mathscr{H}\text{\kern -3pt {\calligra\large om}}\,}

\DeclareMathOperator{\Fun}{\operatorname{Fun}}

\DeclareMathOperator{\Spec}{\operatorname{Spec}}
\DeclareMathOperator{\Spf}{\operatorname{Spf}}

\DeclareMathOperator{\Map}{\operatorname{Map}}
\DeclareMathOperator{\QCoh}{\operatorname{QCoh}}

\DeclareMathOperator{\G}{\mathbf G}
\DeclareMathOperator{\Z}{\mathbf Z}
\DeclareMathOperator{\Q}{\mathbf Q}

\DeclareMathOperator{\Mod}{\operatorname{Mod}}

\renewcommand{\i}{\infty}

\newcommand{\T}{\mathbf T}

\newcommand{\w}{\widehat}

\renewcommand{\i}{\infty}

\makeatletter
\def\@tocline#1#2#3#4#5#6#7{\relax
  \ifnum #1>\c@tocdepth % then omit
  \else
    \par \addpenalty\@secpenalty\addvspace{#2}%
    \begingroup \hyphenpenalty\@M
    \@ifempty{#4}{%
      \@tempdima\csname r@tocindent\number#1\endcsname\relax
    }{%
      \@tempdima#4\relax
    }%
    \parindent\z@ \leftskip#3\relax \advance\leftskip\@tempdima\relax
    \rightskip\@pnumwidth plus4em \parfillskip-\@pnumwidth
    #5\leavevmode\hskip-\@tempdima
      \ifcase #1
       \or\or \hskip 1em \or \hskip 2em \else \hskip 3em \fi%
      #6\nobreak\relax
    \dotfill\hbox to\@pnumwidth{\@tocpagenum{#7}}\par
    \nobreak
    \endgroup
  \fi}
\makeatother

\usepackage{color}   %May be necessary if you want to color links
\usepackage[hypertexnames=false]{hyperref}
\hypersetup{
    linktoc=all,     %set to all if you want both sections and subsections linked
    linkcolor=black,  %choose some color if you want links to stand out
}

\renewcommand{\i}{\infty}
\renewcommand{\o}{\otimes}

\usepackage{epigraph}

\newcommand{\stackspace}{3.5}
\newcommand{\stack}[2][1cm]{\;\tikz[baseline, yshift=.65ex]%
    {\foreach \k [evaluate=\k as \r using (.5*#2+.5-\k)*\stackspace] in {1,...,#2}{%
    \ifodd\k{\draw[<-](0,\r pt)--(#1,\r pt);}%
    \else{\draw[->](0,\r pt)--(#1,\r pt);}\fi
    }}\;}

\usepackage{relsize}
\usepackage[bbgreekl]{mathbbol}
\usepackage{amsfonts}

\DeclareSymbolFontAlphabet{\mathbb}{AMSb}
\DeclareSymbolFontAlphabet{\mathbbl}{bbold}

\newcommand*{\doublerightarrow}[2]{\mathrel{
  \settowidth{\@tempdima}{$\scriptstyle#1$}
  \settowidth{\@tempdimb}{$\scriptstyle#2$}
  \ifdim\@tempdimb>\@tempdima \@tempdima=\@tempdimb\fi
  \mathop{\vcenter{
    \offinterlineskip\ialign{\hbox to\dimexpr\@tempdima+1em{##}\cr
    \rightarrowfill\cr\noalign{\kern.5ex}
    \rightarrowfill\cr}}}\limits^{\!#1}_{\!#2}}}

\usepackage{textcomp}

\begin{document}

\begin{abstract}
We introduce Frobenius-untwists of the variants of topological cyclic homology, following Manam.
Using these, we construct  modifications of the free loop space over the sphere spectrum, and show that they provide even periodic spectral enhancements of the prismatization stacks of Bhatt-Lurie and Drinfeld. 
We identify  the extra structure on prismatization encoded in the even periodic enhancements  with previously-known structures, such as the Breuil-Kisin twists and the Drinfeld formal group.
\end{abstract}

\maketitle

%\setcounter{tocdepth}{2}
%\tableofcontents

%\newpage

\section*{Introduction}

In a broad sense, this paper concerns the connection between \textit{Hodge theory in mixed characteristic}, and \textit{cyclic homology over the sphere spectrum}.

In a narrow sense, this paper constructs \textit{even periodic enhancements} to certain periodization stacks of Bhatt-Lurie \cite{BLb}, \cite{Bhatt F-gauges} and Drinfeld \cite{Drinfeld}, which in turn geometrize the prismatic cohomology of Bhatt-Scholze \cite{BS22}. More precisely, the following is our main result:

\begin{theoremm}[{Theorem \ref{Main Theorem in Text}}]
Let $X$ be a quasi-syntomic $p$-adic formal scheme.
There exist canonical even periodic formal spectral stacks 
\begin{normalfont} $\pounds X, \text{\textlira} X, \$ X,$\end{normalfont} such that their
 underlying classical  stacks  may be identified with the prismatization stacks as
$$
(\begin{normalfont}\pounds\end{normalfont}X)^\heart \,\simeq\, X^{\mathbbl{\Delta}}, \quad (\begin{normalfont}\text{\textlira}\end{normalfont}X)^\heart \,\simeq \,X^{\mathcal N}, \quad
(\begin{normalfont}\$\end{normalfont}X)^\heart \,\simeq\, X^{\mathrm{Syn}}.
$$
On the $\E$-rings of global functions, these recover Frobenius-untwisted versions of the variants of topological cyclic homology
$$
\mathcal O(\begin{normalfont}\pounds\end{normalfont}X)\,\simeq\, \mathrm{TP}^{(-1)}(X), \quad
\mathcal O(\begin{normalfont}\text{\textlira}\end{normalfont}X)\,\simeq\, \mathrm{TC}^{-(-1)}(X), \quad
\mathcal O(\begin{normalfont}\$\end{normalfont}X)\,\simeq\, \mathrm{TC}(X)^\wedge_p.
$$
The Quillen formal group on the even periodic spectral stacks \begin{normalfont} $\pounds X, \text{\textlira} X, \$ X,$\end{normalfont} recover the Drinfled formal group \cite{Drinfeld FG} on the underlying classical stacks $X^{\mathbbl\Delta}, X^{\mathcal N}, X^\mathrm{Syn}$.
\end{theoremm}

In the rest of the introduction, we  explain how the narrow sense fits into the broad sense. That is to say, we explain how the Main Theorem provides a link between Hodge theory and cyclic homology, and place it in the historical context.

\renewcommand*{\thesubsection}{I.\arabic{subsection}}

\subsection{The story in characteristic zero}
A connection between Hodge theory and cyclic homology is classical over a field of characteristic zero $k$. The starting point is
the  celebrated \textit{Hochschild-Kostant-Rosenberg theorem}, which expresses the relative Hochschild homology of a smooth $k$-algebra $A$ over $k$ in terms of differential forms on $A$ as
$$
\mathrm{HH}(A/k)\,\simeq \, \bigoplus_{i\in \Z}\Omega^i_{A/k}[i].
 $$
 The action of the circle group $\T$ on the left-hand side corresponds to the de Rham differential $d_\mathrm{dR}$ on the right-hand side. Recall that the \textit{periodic cyclic homology} and \textit{negative cyclic homology} are defined in terms of the circle action on Hoschschild homology respectively as the Tate construction and the homotopy fixed-points
 $$
 \mathrm{HP}\,:=\,\mathrm{HH}^{\mathrm t\T}, \qquad\quad \mathrm{HC}^-\,:=\,\mathrm{HH}^{\mathrm h\T}.
 $$
 Following Connes, Loday-Quillen, and Feigin-Tsygan, we may express these for a smooth $k$-algebra $A$ in terms of the de Rham complex $\mathrm{dR}_{A/k} = (\Omega^*_{A/k}, d_\mathrm{dR})$ and its Hodge filtration as
 \begin{equation}\label{PreBMS2}
\mathrm{HP}(A/k)\,\simeq \, \bigoplus_{i\in \Z}\mathrm{dR}_{A/k}[2i], \qquad \mathrm{HC}^-(A/k)\,\simeq \, \bigoplus_{i\in \Z}\mathrm{Fil}_\mathrm{Hodge}^i\, \mathrm{dR}_{A/k}[2i].
 \end{equation}
 
Away from characteristic zero, the strongest form of the HKR theorem fails \cite{ABM}. The connection still persists, but in the weaker form of HKR filtrations. This is not a particular point of interest in this paper.

\subsection{Hodge theory in mixed characteristic -- prismatic cohomology}
 The main focus of our paper concerns the setting of ``integral mixed characteristic". More precisely, this concerns $p$-adic formal schemes: formal schemes over $\Z_p$ (hence ``integral", since we are not considering e.g.\ $\Q_p$), where the formal structure is all given exclusively by the $p$-adic topology.

 In this setting, Hodge theory admits a particularly elegant manifestation in the form of the \textit{prismatic cohomology} $\mathrm R\Gamma_{\mathbbl\Delta}$ of \cite{BS22}. This remarkable cohomology \linebreak theory  specializes to virtually all previously-considered cohomology theories for $p$-adic spectral schemes, including \'etale cohomology, crystalline cohomology, $q$-de Rham cohomology, $\mathbf A_\mathrm{inf}$-cohomology. It also comes equipped with extra structure, reminiscent of more familiar objects in Hodge theory:
 \begin{itemize}
\item The \textit{Breuil-Kisin twists} $\{i\}$, prismatic analogues of the Tate twists $(i)$.
\item The \textit{Nygaard filtration} $\mathrm{Fil}^i_\mathcal N$, a prismatic analogue of the Hodge filtration.
 \end{itemize}
 
 The original site-theoretic approach to prismatic cohomology of  \cite{BS22} relies on the algebraic notion of a prism. In the special case of a \textit{quasi-regular semiperfectoid} (\textit{qrsp} for short) ring $R$,  prismatic cohomology is
 given as
 $$
 \mathrm R\Gamma_{\mathbbl \Delta}(\Spf(R))\,\simeq\, \mathbbl\Delta_R
 $$
 in terms of the single prism $\mathbbl\Delta_R$, determined by a universal property. The perhaps somewhat exotic class of qrsp rings plays an important role, because a large class of $p$-adic formal schemes -- namely the quasi-syntomic ones, which includes for instance all smooth $p$-adic formal schemes -- admit covers by qrsp affine opens. The simple form of prismatic cohomology for qrsp rings therefore makes them very convenient for prismatic computations via \textit{quasi-syntomic descent}.

\subsection{Hodge theory via stacks -- prismatization}
In his work on nonabelian Hodge theory \cite{Simpson}, Simpson introduced the \textit{de Rham stack} $X^\mathrm{dR}$ of a scheme $X$ over a field $k$ of characteristic zero. It encodes the de Rham cohomology of $R$ as its $\E$-ring of global functions  $\mathcal O(X)\, \simeq \, \mathrm R\Gamma_\mathrm{dR}(X/k)$, but its geometry is useful in the study of $X$, e.g.\ quasi-coherent sheaves on it correspond to D-modules on $X$, etc. The  Hodge filtration of the de Rham cohomology is encoded by a further stack $X^{\mathrm{dR}, +}\to \mathbf A^1/\mathbf G_m$ (classically denoted $X^\mathrm{Hod}$), which has $X^\mathrm{dR}$ as its generic fiber, and the shifted tangent bundle $\mathrm T[1]X \simeq \mathrm B_X\mathrm TX$ as its special fiber. A characteristic $p$ variant of the stack $X^\mathrm{dR}$ was found by Drinfeld in \cite{Drinfeld stack}.

Inspired by this, Bhatt-Lurie \cite{BLb}, \cite{Bhatt F-gauges} and Drinfeld \cite{Drinfeld} introduced the analogous stacky approach to prismatic cohomology. To a $p$-adic formal scheme $X$, they associate the trio of \textit{prismatization stacks}
$$
X\mapsto X^{\mathbbl\Delta},\, X^{\mathcal N},\, X^\mathrm{Syn},
$$
respectively called the \textit{prismatization of $X$}, the \textit{filtered prismatization of $X$}, and the \textit{syntomification of $X$}.
They geometrically encode the properties of and structure on prismatic cohomology. For instance, the global functions on prismatization recover prismatic cohomology as $\sO(X^{\mathbbl\Delta})\,\simeq\, \mathrm R\Gamma_{\mathbbl\Delta}(X)$, and a canonical map $X^{\mathcal N}\to \mathbf A^1/\mathbf G_m$ has $X^{\mathbbl \Delta}$ as the generic fiber, and it  encodes the Nygaard filtration.

The prismatization stacks are defined in \cite{Bhatt F-gauges} by \textit{transmutation} -- for a fixed $?\in \{\mathbbl\Delta, \mathcal N, \mathrm{Syn}\}$, this works roughly by first defining a (derived) ring stack $\mathbf G_a^{?}$ ``by hand", and then setting the $?$-prismatization of any $p$-adic formal scheme $X$ to be given in terms of the functor of points as $X^?(R) \,:=\, X(\G_a^?(R))$. But in the special case of a qrsp affine $X=\Spf(R)$, the first two of these stacks may be expressed explicitly through the simple formulas
\begin{equation}\label{Prism stacks via qsyn desc}
\Spf(R)^{\mathbbl\Delta}\,\simeq \,\Spf(\mathbbl\Delta_R), \qquad \Spf(R)^{\mathcal N}\,\simeq \, \Spf\big(\bigoplus_{i\in \mathbf Z}\mathrm {Fil}^i\,\mathbbl\Delta_R\big)/\mathbf G_m,
\end{equation}
and $\Spf(R)^{\mathrm{Syn}}$ is a certain coequalizer of maps between the former two. This gives a method for determining the prismatization stacks of any quasi-syntomic $p$-adic formal scheme via quasi-syntomic descent.

\subsection{The connection to topological cyclic homology}
\textit{Topological Hochschild homology} $\mathrm{THH}$ is the traditional name for the relative Hochschild homology $\mathrm{HH}(-/\mathbf S)$ over the sphere spectrum $\mathbf S$. By passing respectively to the homotopy fixed-points and the Tate construction with respect to the circle group action on $\mathrm{THH}$, we obtain the \textit{topological negative cyclic homology}  and
\textit{topological periodic cyclic homology}
$$
\mathrm{TC}^-\,:=\,\mathrm{THH}^{\mathrm h\T},\qquad\quad\mathrm{TP}\,:=\,\mathrm{THH}^{\mathrm t\T}.
$$
A further variant is the \textit{topological cyclic homology} $\mathrm{TC}$, which can be identified (at least $p$-typically) described as an equalizer of two maps from $\mathrm{TC}^-$ to $\mathrm{TP}$.

The landmark connection between these variants of topological cyclic homology and prismatic cohomology, somewhat reminiscent of the identifications \eqref{PreBMS2}, was made by Bhatt-Morrow-Scholze:

\renewcommand*{\thetheoremmm}{I.\arabic{theoremmm}}

\begin{theoremmm}[{\cite{BMS2}}]\label{Thm BMS2 iso}
Let $R$ be a qrsp ring. Then the $\E$-rings $\mathrm{TP}(R)^\wedge_p$ and $\mathrm{TC}^-(R)^\wedge_p$ are even, and their homotopy groups may be  naturally identified in terms of prismatic cohomology as
$$
\pi_{2i}(\mathrm{TP}(R)^\wedge_p)\, \simeq \, \widehat{\mathbbl \Delta}_R\{i\}, \qquad
\pi_{2i}(\mathrm{TC}^-(R)^\wedge_p)\, \simeq \, \mathrm{Fil}_{\mathcal N}^{ i}\,\widehat{\mathbbl \Delta}_R\{i\}.
$$
\end{theoremmm}
The topological cyclic homology $\mathrm{TC}(R)^\wedge_p$ also obtains an interpretation in terms of the \textit{syntomic complexes} $\mathbf Z_p(i)(R)$. Note however that in all these comparisons, the right-hand side bears the over-script $\widehat{(-)}$ to indicate that we have passed to the completion with respect to the Nygaard filtration.

\subsection{The goals and methods of this paper}
Our Main Theorem, as stated at the start of this introduction, is a far-reaching extension of the Bhatt-Morrow-Scholze result recalled above as Theorem \ref{Thm BMS2 iso}. It simultaneously extends it in the following three directions:
\begin{enumerate}
\item \label{1} It removes the Nygaard-completions, i.e.\ replaces $\widehat{\mathbbl\Delta}_R$ with $\mathbbl\Delta_R$.
\item \label{2} It relaxes the both the qrsp affineness assumptions, to be valid instead for any quasi-syntomic $p$-adic formal scheme.
\item \label{3} It recovers not only prismatic cohomology, but the prismatization stacks also.
\end{enumerate}

Our means of achieving all this at the same time is to provide \textit{even periodic enhancements} of the prismatization stacks. 
To elaborate on what that means, let us attempt a quick review of \textit{spectral algebraic geometry}, directing the reader to Lurie's authoritative work on the subject \cite{SAG} for more details in general, or to \cite{Even periodization} for the precise \textit{functor of points} setup we use in particular.

We understand \textit{spectral stacks}  to be certain kinds of functors $\mX : \CAlg\to\mathrm{Ani}$ from the $\i$-category of (not-necessarily connective) $\E$-ring spectra to anima. The role of \textit{affines} is played by the corepresentable functors $\Spec(A) : B\mapsto \Map_{\CAlg}(A, B).$ The constructions $\mX\mapsto \mX^\heart$ and $\mX\mapsto \sO(\mX)$ of the \textit{underlying classical stack} and the \textit{$\E$-ring of global functions} respectively, are defined for arbitrary spectral stacks by descent from affines, where they are in turn set to be
$$
\Spec(A)^\heart\,\simeq \,\Spec(\pi_0(A)), \qquad \sO(\Spec(A))\,\simeq \, A.
$$
A spectral stack is \textit{even periodic} if it is generated by affines of the form $\Spec(A)$ where $A$ is an even periodic $\E$-ring, i.e.\ $\pi_{2i}(A)\simeq \pi_0(A)$ and $\pi_{2i+1}(A)\simeq 0$ for all $i\in\Z$. For any spectral stack $\mX$, its \textit{even periodization} is the universal even periodic spectral stack $\mX^\mathrm{evp}$ together with a map of spectral stacks $\mX^\mathrm{evp}\to \mX$.

\subsection{Even periodic enhancements}
With this understanding of the language of spectral algebraic geometry, we say that an \textit{even periodic enhancement} of a classical stack $\mX_0$ is any even periodic spectral stack $\mX$ with the underlying classical stack equivalent to $\mX^\heart\,\simeq \, \mX_0$.

The search for and study of even periodic enhancements of various classical stacks has a rich and storied history in spectral algebraic geometry. As an essential step in the construction of topological modular forms, Goerss-Hopkins-Miller produced an even periodic enhancement of the classical moduli stack of elliptic curves. The celebrated Morava E-theory in chromatic homotopy theory provides precisely the even periodic enhancement of the classical Lubin-Tate deformation spaces of formal groups. Inspired by Lurie recasting both of the above examples explicitly in terms of spectral algebraic geometry in \cite{Elliptic 2}, Davies in \cite{Davies} and the author in \cite{ChromaticCartoon}
constructed even periodization of the classical moduli stacks of $p$-divisible groups and formal groups respectively. As we show in \cite{ChromaticFiltration}, a large chunk of chromatic homotopy theory may be interpreted as concerning the geometry of this latter even periodic enhancement.

\subsection{A sketch of the proof of the Main Theorem}
Our approach to the proof of the Main Theorem is to first define the even periodic enhancements in question for qrsp affines, and then extend via quasi-syntomic descent. For qrsp rings $R$, this amounts to \textit{Nygaard-decompleting} all the variants of topological cyclic homology. Following \cite{Deven}, we define the \textit{Frobenius-untwists}
$\mathrm{TP}^{(-1)}(R)$, $\mathrm{TC}^{-(-1)}(R)$ by a localization procedure (in the Appendix, we show how they can also be obtained \textit{by animation}, as explained to us by Devan Manam), and show that they do the job. Then we define the functors
$$
X\mapsto \pounds X, \, \text{\textlira} X, \,\$ X 
$$
from quasi-syntomic $p$-adic formal schemes to spectral stacks, respectively called the \textit{spherical Tate-loop space on $X$}, the \textit{spherical Nygaard-loop space on $X$}, and the \textit{spherical prismatization of $X$}, via quasi-syntomic descent. This reduces the definition to the case of a qrsp affine $X=\Spf(R)$, in which case we set
$$
\pounds\Spf(R)\,\simeq \, \Spf\big(\mathrm{TP}^{(-1)}(R)\big),\qquad \text{\textlira}\Spf(R)\,\simeq \, \Spf\big(\mathrm{TC}^{-(-1)}(R)\big)^\mathrm{evp},
$$
where $(-)^\mathrm{evp}$ denotes the even periodization of \cite{Even periodization}. Finally, we take $\$\Spf(R)$ to be the coequalizer of a certain pair of maps $\pounds\Spf(R)\rightrightarrows \text{\textlira}\Spf(R)$. 

With the so-defined even periodic stacks $X\mapsto \pounds X, \, \text{\textlira} X, \,\$ X$, the identification of their underlying classical stacks with the prismatization stacks $X^{\mathbbl\Delta}, \, X^{\mathcal N}, \, X^{\mathrm{Syn}}$ then follows from the explicit description  \eqref{Prism stacks via qsyn desc} of the latter.

\subsection{Consequences of the Main Theorem}
We conclude the  paper by showing how  some of the familiar structure on the prismatization stacks is encoded in the even periodic enhancements. For instance, the even homotopy sheaves 
$$
\pi_{2i}(\sO_{\pounds X}),\, \pi_{2i}(\sO_{\textit{\textlira} X}),\, \pi_{2i}(\sO_{\$ X})
$$
may be recognized for all $i\in\Z$ as the Breuil-Kisin twists
$$
\sO_{X^{\mathbbl\Delta}}\{i\},\,  \sO_{X^{\mathcal N}}\{i\},\, \sO_{X^{\mathrm{Syn}}}\{i\}.
$$
More spectacularly, we expand on the work of Manam in \cite{Deven} to show that the Quillen formal groups, which exist on the spectral stacks $X\mapsto \pounds X,  \text{\textlira} X,\$ X$ merely by the virtue of their even periodicity, induce upon the underlying classical stacks $X^{\mathbbl\Delta},  X^{\mathcal N},  X^{\mathrm{Syn}}$ precisely the Drinfeld formal group, exhibited on the prismatization stacks in \cite{Drinfeld FG}.

\subsection{The connection to loop spaces}
As the name and notation both suggest, we would like to encourage viewing the spherical Tate-loop and Nygaard-loop spaces $\pounds X, \text{\textlira}X$ as variants of the spherical free loop space
$$
\mathcal LX \,\simeq\, \Map_{\mathrm{SpStk}}(\T, X)\,\simeq\, X\times_{X\times_{\Spec(\mathbf S)}X}X,
$$ 
which geometrizes $\mathrm{THH}$. In \cite[Section 6.4]{Even periodization}, we showed how a coarse quotient of the even periodization of the spherical free loop space $(\mathcal LX)^\mathrm{evp}/\!/\T$ gives an even periodic enhancement of Nygaard-complete filtered prismatization $X^{\widehat{\mathcal N}}$. The two even periodic enhancements of filtered prismatization are connected by a canonical map
$$
(\mathcal LX)^\mathrm{evp}/\!/\T\to \text{\textlira X},
$$
which induces the Nygaard-completion map $X^{\widehat{\mathcal N}}\to X^{\mathcal N}$ on the underlying classical stacks.

\subsection{Relationship to other work}
This work has some non-trivial overlap with the upcoming independent work of Hahn-Devalapurkar-Raksit-Senger-Yuan, some accounts of which can be found in \cite{Equivariant aspects of de-completing cyclichomology}, \cite{Sanath's Thesis}, as well as \cite{Jeremy's talk} and \cite{Arpon's talk}. Their technical setup is slightly different, for instance eschewing the \textit{even periodic} stacks we consider for a different notion of \textit{even stacks}. Unlike ours, their discussion of spectrally-enhanced prismatization falls under the general heading of studying a circle-equivariant even filtration for cyclonic ring spectra. As such, it may be applied to a variety of objects outside the realm of $p$-adic geometry, including the sphere spectrum and complex bordism, which the technology of the present paper certainly can not accommodate. In short, we believe  it a fair assessment that relative to Hahn-Devalapurkar-Raksit-Senger-Yuan, the present paper is less ambitious, but correspondingly also simpler. We hope that this, as well as the objective desirability of a plurality of perspectives, justifies the coexistence of both papers.

\subsection*{Acknowledgments}
I owe thanks to several people: first to Akhil Mathew, for bringing prismatic cohomology to my attention on a visit to the University of Chicago in the late Fall of 2021. To Deven Manam for generously sharing his ideas about the Drinfeld formal group and Nygaard-decompletion with me.
To Jermey Hahn, for patiently explaining to me the results and scope of his upcoming joint work with Devalapurkar-Raksit-Senger-Yuan.
To David Ben-Zvi, David Gepner, and Sam Raskin for their support and many stimulating conversations about this project. 
And finally to Thomas Nikolaus, whose suggestion that  equivariance should not be necessary for spectral prismatization at a picnic in the Summer of 2025 was the unsuspected  missing cog that made the material of this paper slide into place.

\section{Frobenius-untwisted topological cyclic homology}\label{Section 1}
The quintessential first step is to remove the Nygaard-completions appearing in the Bhatt-Morrow-Scholze Theorem's result, summarized above as Theorem \ref{Thm BMS2 iso}. This necessitates a modification of the variants of topological cyclic homology, in which we follow Manam from \cite[Section 1]{Deven}.

In what follows, and throughout this paper, we let $\mathrm{CAlg}^\mathrm{qrsp}_{\mathbf Z_p}$  denote the category of qrsp (quasi-regular semiperfectoid) rings, always implicitly $p$-complete and classical, see \cite[Definition 4.20]{BMS2}.

\begin{theorem}[{\cite{Deven}}]\label{Theorem Frobenius untwist for TP}
There is a canonical functor $\mathrm{TP}^{(-1)}:\mathrm{CAlg}^\mathrm{qrsp}_{\mathbf Z_p}\to \CAlg^\mathrm{evp}_\mathrm{ad}$ with a natural transformation
$$\mathrm{TP}^{(-1)}(R)\to \mathrm{TP}(R)^\wedge_p$$
such that for any $i\in \mathbf Z$, its composite with the functor $\pi_{2i}:\CAlg^\mathrm{evp}\to\mathrm{CAlg}^\heart$ recovers the completion natural transformation
$$
\widehat{\mathbbl \Delta}_R\{i\}\to \mathbbl \Delta_R \{i\}.
$$
\end{theorem}

Since this construction is playing such an outsized role in the present paper, let us repeat here its proof from \cite{Deven}, emphasizing the functoriality of the construction.

\begin{proof}[Proof]
By the theory of prismatic cohomology, e.g.\ \cite[Remark 5.7.10]{BLa}, the divided Frobenius $\varphi$ admits a factorization through the Nygaard-completion as
\begin{equation}\label{Factoring divided Frobenius on homotopy}
\xymatrix{
\mathrm{Fil}_\mathcal N^i\,\mathbbl \Delta_R\{i\}\ar[dr]\ar[rr]^{\varphi\{i\}} & & \mathbbl\Delta_R\{i\},\\
& \mathrm{Fil}_\mathcal N^i\,\widehat{\mathbbl \Delta}_R\{i\}\ar[ur]_{\widetilde{\varphi}\{i\}} &
}
\end{equation}
naturally in terms of the qrsp ring $R$ and $i\in \mathbf Z$. If we let
$$
\mathscr S_R\, \subseteq\, \bigoplus_{i\in \mathbf Z}\mathrm{Fil}_\mathcal N^i\,\widehat{\mathbbl \Delta}_R\{i\} \,\cong \,\pi_*(\mathrm{TC}^-(R)^\wedge_p)
$$
denote the graded multiplicative subset of all those elements which become invertible in $\bigoplus_{i\in \mathbf Z}\mathbbl\Delta_R\{i\}$ after applying $\widetilde{\varphi}$, the desired $\E$-ring is obtained as the localization
$$
\mathrm{TP}^{(-1)}(R) \,:=\,\big(\mathrm{TC}^-(R)^\wedge_p[\mathscr S_R^{-1}]\big)^\wedge_{(p, I)},
$$
completed with respect to the ideal $(p, I)$, where $I\subseteq\mathbbl\Delta_R$ is the prismatic ideal. It then follows by design that the $\E$-ring $\mathrm{TP}^{(-1)}(R)$ is even and has homotopy groups
$$
\pi_{2i}\big(\mathrm{TP}^{(-1)}(R)\big)\simeq \mathbbl \Delta_R\{i\},
$$
For a proof that this is indeed a graded localization, see \cite[Proposition 1.2]{Deven}.
Since the Tate-Frobenius of Nikolaus-Scholze  $\varphi_{\mathrm{NS}}$ induces the (Nygaard-completed) divided Frobenius $\widehat\varphi$
under the identifications of Theorem \ref{Thm BMS2 iso},
there is by design also a  commutative diagram of even $\E$-rings
\begin{equation}\label{Factoring the Tate Frobenius}
\xymatrix{
\mathrm{TC}^-(R)^\wedge_p\ar[dr]_{{\varphi}^{(-1)}}\ar[rr]^{\varphi_\mathrm{NS}} & & \mathrm{TP}(R)^\wedge_p,\\
& \mathrm{TP}^{(-1)}(R)\ar[ur] &
}
\end{equation}
which induces upon the passage to $\pi_{2i}$ a factorization of the (Nygaard-completed) divided Frobenius $\widehat{\varphi}$ through the above-described lifted Frobenius $\widehat{\varphi}$.
$$
\xymatrix{
\mathrm{Fil}_\mathcal N^i\,\widehat{\mathbbl \Delta}_R\{i\}\ar[dr]_{\widetilde{\varphi}\{i\}}\ar[rr]^{\widehat{\varphi}\{i\}} & & \widehat{\mathbbl\Delta}_R\{i\}.\\
& {\mathbbl \Delta}_R\{i\}\ar[ur], &
}
$$
Examining the respective lower right maps in these two diagrams gives the last claim in the statement of the Theorem.

It remains to verify the functoriality and naturality claims concerning $\mathrm{TP}^{(-1)}$. The naturality will follow from the naturality of localization, so it suffices to show  that the construction $R\mapsto \mathrm{TP}^{(-1)}(R)$ extends to a functor $\mathrm{CAlg}^\mathrm{qrsp}_{\mathbf Z_p}\to\CAlg_\mathrm{ad}$. For this purpose, let $\CAlg^{++}$ (resp.\ $\CAlg^+$) denote the $\i$-category consisting of triples $(A, \mathscr S, J)$ (resp.\ pairs $(A, J)$) of an $\E$-ring $A$, a multiplicative subset $\mathscr S\subseteq \pi_*(A)$, and a finitely generated ideal $J\subseteq \pi_0(A)$, and evident structure-preserving maps between them. We wish to construct the functor $\mathrm{TP}^{(-1)}$ as the composite
$$
\mathrm{CAlg}^\mathrm{qrsp}_{\mathbf Z_p}\xrightarrow{\mathrm{TC}^-_{++}} \CAlg^{++}\xrightarrow{(-)[\mathscr S^{-1}]} \CAlg^+\xrightarrow{(-)^\wedge_J} \CAlg_\mathrm{ad},
$$
where the second and third functors, i.e.\ localization and completion respectively, are clear. The first functor $\mathrm{TC}^-_{++} :\mathrm{CAlg}^\mathrm{qrsp}_{\mathbf Z_p}\to\CAlg^{++}$ is given by
$$
R\mapsto (\mathrm{TC}^-(R), \mathscr S_R, (p, I)).
$$
Since $\mathrm{TC}^-$ is evidently functorial, whereas $\mathscr S_R$ and $(p, I)$ are $\pi_*$-level data, this is no longer a homotopy coherence question, but rather a simple observation about prismatic cohomology. Namely, let $R\to R'$ be a map of qrsp rings, inducing maps $\mathbbl\Delta_R\to \mathbbl \Delta_{R'}$, as well as between the Nygaard filtrations and Nygaard-completions. From the naturality of the divided Frobenii, it follows that the diagram
$$
\xymatrix{
\mathrm{Fil}_\mathcal N^i\,\widehat{\mathbbl \Delta}_R\{i\}\ar[d]\ar[rr]^{\widetilde{\varphi}\{i\} } & &\mathbbl{\Delta}_R\{i\}\ar[d]\\
\mathrm{Fil}_\mathcal N^i\,\widehat{\mathbbl \Delta}_{R'}\{i\}\ar[rr]^{\widetilde{\varphi}\{i\} } & &\mathbbl{\Delta}_{R'}\{i\}
}
$$
commutes, from which it is clear that both the prismatic ideals are preserved by the horizontal maps, as well as that the left horizontal map maps $\mathscr S_R$ to $\mathscr S_{R'}$, since invertible elements are preserved under any ring map. 
\end{proof}

\begin{remark}
Theorem \ref{Theorem Frobenius untwist for TP} guarantees that the $\E$-ring $\mathrm{TP}^{(-1)}(R)$ comes equipped with a canonical adic structure for every qrsp ring $R$. Under the isomorphism $\pi_0\left(\mathrm{TP}^{(-1)}(R)\right)\simeq \mathbbl \Delta_R$, this adic topology may be identified as having $(p, I)$ for its ideal of definition, where $I\subseteq \mathbbl\Delta_R$ is the prismatic ideal.
\end{remark}

\begin{remark}\label{Remark 1.3}
Let us expand on how the above construction is ``Frobenius-untwisted". 
Let $R$ be a qrsp ring, and $P$ a perfectoid ring together with a  surjection $P\to R$. This allows us to  identify absolute prismatic complex with the relative prismatic complex $\mathbbl\Delta_R \simeq \mathbbl\Delta_{R/P}$.We also chose a generator $d$ for the prismatic ideal of the perfect prism $\mathbbl\Delta_P=\mathbf A_\mathrm{inf}(P) =\mathrm W(P^\flat)$, and therefore also of $\mathbbl\Delta_R$. This in particular specifies a trivialization of the Breuil-Kisin twists, which we therefore suppress.

To keep   track of the $\mathbf A_\mathrm{inf}(P)$-algebra structure, the $\delta$-ring Frobenius map must be written as $\phi:\tfrac 1{d^i}\mathrm{Fil}^i_{\mathcal N}\, \mathbbl\Delta_{R}^{(1)}\to \mathbbl\Delta_{R}$, i.e.\ with a Frobenius twist over $\mathbbl\Delta_P$. In terms of subsets of the localization $\mathbbl\Delta_R[\tfrac 1 d]$, the relative Frobenius is further compatible with the sequence of inclusions
\begin{equation}\label{diagram F}
\mathbbl\Delta^{(1)}_{R}\,=\,\mathrm{Fil}^0_{\mathcal N}\, \mathbbl\Delta^{(1)}_{R} \,\subseteq\, \tfrac 1d\mathrm{Fil}^1_{\mathcal N}\, \mathbbl\Delta^{(1)}_{R}\,\subseteq\, \tfrac 1{d^2}\mathrm{Fil}^2_{\mathcal N}\, \mathbbl\Delta_{R}^{(1)}\,\subseteq\,\cdots \xrightarrow{\phi}\mathbbl\Delta_R,
\end{equation}
and by \cite[Corollary 5.2.16]{BLa} this directed system induces an equivalence
\[
\big(\varinjlim_i \tfrac 1{d^{i}}\mathrm{Fil}_{\mathcal N}^{i}\mathbbl\Delta_{R}^{(1)}\big)^\wedge_{(p, d)} \,\simeq \mathbbl\Delta_{R}.
\]
The divided Frobenius $\varphi$ is given in terms of the prismatic $\delta$-ring Frobenius $\phi$ on the $i$-th piece of the Nygaard filtration by $\varphi =\frac \phi{d^i}$. By using the divided Frobenius, diagram \ref{diagram F} may be rewritten as
\[
\mathbbl\Delta^{(1)}_{R}\,=\,\mathrm{Fil}^0_{\mathcal N}\, \mathbbl\Delta^{(1)}_{R} \,\xrightarrow{d}\, \mathrm{Fil}^1_{\mathcal N}\, \mathbbl\Delta^{(1)}_{R}\,\xrightarrow{d}\,\mathrm{Fil}^2_{\mathcal N}\, \mathbbl\Delta_{R}^{(1)}\,\xrightarrow{d}\,\cdots \xrightarrow{\varphi}\mathbbl\Delta_R.
\]
The observation \eqref{Factoring divided Frobenius on homotopy} that the divided Frobenius always factors canonically through the Nygaard-completion, we obtain yet another diagram
\[
\widehat{\mathbbl\Delta}^{(1)}_{R}\,=\,\mathrm{Fil}^0_{\mathcal N}\, \widehat{\mathbbl\Delta}^{(1)}_{R} \,\xrightarrow{d}\, \mathrm{Fil}^1_{\mathcal N}\, \widehat{\mathbbl\Delta}^{(1)}_{R}\,\xrightarrow{d}\,\mathrm{Fil}^2_{\mathcal N}\, \widehat{\mathbbl\Delta}_{R}^{(1)}\,\xrightarrow{d}\,\cdots \xrightarrow{\widetilde{\varphi}}{\mathbbl\Delta}_R,
\]
which turns out to be a colimit diagram in the category of $(p, d)$-complete modules. As such, the completed colimit
\begin{equation}\label{diagram d varphi}
\varinjlim\big(\widehat{\mathbbl\Delta}_{R}\, \xrightarrow{\phi^{-1}(d)}\, \mathrm{Fil}^1_{\mathcal N}\, \widehat{\mathbbl\Delta}_{R}\,\xrightarrow{\phi^{-1}(d)}\,\mathrm{Fil}^2_{\mathcal N}\, \widehat{\mathbbl\Delta}_{R}\,\xrightarrow{\phi^{-1}(d)}\,\mathrm{Fil}^3_{\mathcal N}\, \widehat{\mathbbl\Delta}_{R}\,\xrightarrow{\phi^{-1}(d)}\,\cdots \big)^\wedge_{(p, d)}
\end{equation}
exhibits through the divided Frobenius map $\widetilde{\varphi}$ a ring $\mathbbl\Delta_R^{(-1)}$, which simultaneously Nygaard-decompletes and Frobenius-untwists the prismatic complex $\mathbbl\Delta_R$. Under the identifications of Bhatt-Morrow-Scholze, recalled above as Theorem \ref{Thm BMS2 iso},
or more precisely by \cite[Proposition 6.2]{BMS2}, the $p$-completed negative topological cyclic homology of the perfectoid ring $P$ has homotopy groups given by\footnote{The formula appearing in \cite[Proposition 6.2]{BMS2} technically lists the quotiented relation (when rewritten with our notational choices) as $ut-d$. The reason for this inconsistency is that \cite{BMS2} is in general imprecise about keeping track of Frobenius or Breuil-Kisin twists -- a harmless laxity when working over a perfectoid base. But being more precise, e.g.\ \cite[Example 5.5.6]{Bhatt F-gauges}, the Nygaard filtration on $\mathbbl\Delta_R$ for a qrsp ring $R$ does indeed feature $ut-\phi^{-1}(d).$}
$$
\pi_*\big(  \mathrm{TC}^-(P)^\wedge_p\big)\,\simeq \, \mathbf A_\mathrm{inf}(P)[u, t]/(ut-\phi^{-1}(d))
$$
where $u$ is in homotopy degree $2$ and $t$ in homotopy degree $-2$.
We see by \cite[Proposition 6.3]{BMS2} that the multiplication by $u$ induces the multiplication by $\phi^{-1}(d)$ on homotopy groups in light of the equivalences of Theorem \ref{Thm BMS2 iso}. It follows that the colimit \eqref{diagram d varphi} is induced on $\pi_0$ by the telescopic limit formula for the $\E$-ring localization\footnote{The approach of Manam from \cite{Deven}, which we have follows in the proof of Theorem \ref{Theorem Frobenius untwist for TP}, amounts ostensibly to phrasing this same construction in a way which is evidently independent of the choice of $P$. But in the form discussed in this Remark -- which can also be found in \cite[proof of Proposition 1.2]{Deven} -- the author had first learned of this construction from Akhil Mathew.}
\begin{equation}\label{Formula invert u}
\mathrm{TC}^{-1}(R)^\wedge_p[u^{-1}]^\wedge_{(p, d)} \,\simeq \, \mathrm{TP}^{(-1)}(R).
\end{equation}
The latter identification comes from noting that in the case in question, the graded localizing subset $\mathscr S_R\subseteq\pi_0(\mathrm{TC}^{-1}(R))$ from the proof of Theorem \ref{Theorem Frobenius untwist for TP} is generated by the element $u$.
The homotopy groups of the  $\E$-ring $\mathrm{TP}^{(-1)}(R)$ therefore indeed compute precisely the colimit \eqref{diagram d varphi}, which is to say, the Frobenius untwist. But since the choice of where the Frobenius twists start counting from is arbitrary -- after all, we are working over the perfect prism $\mathbbl\Delta_P\simeq \mathbf A_\mathrm{inf}(P)$, whose Frobenius is an isomorphism, -- we elect to re-index so that the Frobenius twists disappear from notation e.g.\ in the statement of Theorem \ref{Theorem Frobenius untwist for TP}.
\end{remark}

\begin{remark}
Continuing with the setting and notation of Remark \ref{Remark 1.3}, note that the element $t\in \pi_{-2}\big(\mathrm{TC}^-(R)^\wedge_p\big)$ is the standard ``fundamental class of the circle", appearing since we took homotopy $\T$-fixed-points. Inverting this class is one explicit description of the Tate construction for a $\T$-action. That is to say, there is a canonical equivalence of $\E$-rings
$$
\mathrm{TC}^{-1}(R)^\wedge_p[t^{-1}]^\wedge_{(p, d)} \,\simeq \, \mathrm{TP}(R)^\wedge_p.
$$
Contrasting this with \eqref{Formula invert u}, we see that the choice of which generator of the homotopy ring to invert matters a fair deal in this case.
\end{remark}

Using the above-defined Frobenius-untwists $\mathrm{TP}^{(-1)}$ of (the $p$-completion of) the topological periodic cyclic homology $\mathrm{TP}$, we can similarly define a Frobenius-untwist (as well as Nygaard-decompleted) version  $\mathrm{TC}^{-(-1)}$ of the topological negative cylic homology $\mathrm{TC}^-$. The author is informed by Deven Manam that this construction was independently suggested by Arpon Raksit.

\begin{cons}\label{Cons untwisted TC^-}
Let $R$ be a qrsp ring. Let
$$
\xymatrix{
\mathrm{TC}^{-(-1)}(R) \ar[rr]^{\mathrm{can}}\ar[d] & &\mathrm{TP}^{(-1)}(R)\ar[d] \\
\mathrm{TC}^-(R)^\wedge_p\ar[rr]^{\mathrm{can}}& &\mathrm{TP}(R)^\wedge_p
}
$$
be a pullback square in the $\i$-category of adic $\E$-rings, where the lower horizontal arrow is the canonical map and the right vertical arrow is the natural transformation from Theorem \ref{Theorem Frobenius untwist for TP}. It follows from the functoriality therein that this defines a functor
$$
\mathrm{TC}^{-(-1)}:\mathrm{CAlg}^\mathrm{qrsp}_{\mathbf Z_p}\to \CAlg_\mathrm{ad},
$$
such that the above pullback square is natural in $R$.
\end{cons}

\begin{prop}\label{Prop homotopy groups of untiwsted TC^-}
Let $R$ be a qrsp ring. Then the $\E$-ring $\mathrm{TC}^{-(-1)}(R)$ is even, and for any $i\in \mathbf Z$ there is a canonical isomomorphisms
$$
\pi_{2i}\big(\mathrm{TC}^{-(-1)}(R)\big)\,\simeq\, \mathrm{Fil}^i_{\mathcal N}\, \mathbbl \Delta_R\{i\}.
$$
\end{prop}

\begin{proof}
In light of Construction \ref{Cons untwisted TC^-}, this amounts to the assertion that the canonical square
$$
\xymatrix{
\mathrm{Fil}^i_{\mathcal N}\, \mathbbl \Delta_R\{i\} \ar[d] \ar[rr]^{\subseteq}& &\mathbbl\Delta_R\{i\}\ar[d] \\
\mathrm{Fil}^i_{\mathcal N}\, \widehat{\mathbbl \Delta}_R\{i\}\ar[rr]^{\subseteq}& & \widehat{\mathbbl \Delta}_R\{i\}
}
$$
is a pullback for all $i\in \mathbf Z$. But this can even be taken as the definition of the Nygaard filtration on $\mathbbl\Delta_R$, see \cite[Construction 7.11]{Antieau-Krause-Nikolaus}.
\end{proof}

\begin{cons}\label{Cons of div Frob}
The natural transformation $\varphi:\mathrm{TC}^{-(-1)}\to \mathrm{TP}^{(-1)}$, given for every qrsp ring $R$ by the composite
$$
\mathrm{TC}^{-(-1)}(R)\to \mathrm{TC}^-(R)^\wedge_p\xrightarrow{{\varphi}^{(-1)}}\mathrm{TP}^{(-1)}(R),
$$
where the second map is the first part of the factorization \eqref{Factoring the Tate Frobenius} of the Nikolaus-Scholze Tate-Frobenius $\varphi_\mathrm{NS} : \mathrm{TC}^-(R)^\wedge_p\to \mathrm{TP}(R)^\wedge_p$.
\end{cons}

\begin{prop}\label{Prop Frob on homotopy}
Let $R$ be a qrsp ring.
The $\E$-ring map
$$
\varphi:\mathrm{TC}^{-(-1)}(R)\to\mathrm{TP}^{(-1)}(R)
$$
of Construction \ref{Cons of div Frob} induces for every $i\in \mathbf Z$ upon the passage to the homotopy groups $\pi_{2i}$ the (Breuil-Kisin-twisted) divided Frobenius morphism
$$
\varphi\{i\} : \mathrm{Fil}_\mathcal N^i\,\mathbbl \Delta_R\{i\}\to \mathbbl \Delta_R\{i\}.
$$
\end{prop}

\begin{proof}
By construction, the commutative triangle
$$
\xymatrix{
\mathrm{TC}^{-(-1)}(R)\ar[dr]\ar[rr]^{\varphi} & & \mathrm{TP}^{(-1)}(R),\\
&\mathrm{TC}^-(R)^\wedge_p\ar[ur]_{\varphi^{(-1)}} &
}
$$
induces upon the passage to $\pi_{2i}$ the commutative triangle \eqref{Factoring divided Frobenius on homotopy} from the proof of Theorem \ref{Theorem Frobenius untwist for TP}, proving the desired claim.
\end{proof}

One might at this point expect to define a Frobenius-untwisted version of $\mathrm{TC}$, topological cyclic homology. But as it turns out, this yields nothing new.

\begin{prop}\label{Prop TC don't care}
For any qrsp ring $R$, there is a canonical equivalence of $\E$-rings
$$
\mathrm{TC}(R)^\wedge_p\,\simeq \, \mathrm{fib}\big(\mathrm{TC}^{-(-1)}(R)\xrightarrow{\varphi\,-\,\mathrm{can}}\mathrm{TP}^{(-1)}(R)\big).
$$
\end{prop}

\begin{proof}
Let us temporarily denote the right-hand side by $\mathrm{TC}^{(-1)}(R)$. Expressing the fiber or a difference as an equalizer and the latter as a pullback square, we have the commutative diagram
$$
\xymatrix{
\mathrm{TC}^{(-1)}(R)\ar[rr] \ar[d]^{} & & \mathrm{TP}^{(-1)}(R)\ar[d]^{\Delta} \\ \mathrm{TC}^{-(-1)}(R)^\wedge_p\ar[rr]^{(\mathrm{can}, \,\varphi)\qquad} \ar[d]^{} & & \mathrm{TP}^{(-1)}(R)\times \mathrm{TP}^{(-1)}(R) \ar[d]^{}\\
\mathrm{TC}^-(R){}^\wedge_p\ar[rr]^{(\mathrm{can},\, \varphi^{(-1)})\qquad} & & \mathrm{TP}(R)^\wedge_p\times \mathrm{TP}^{(-1)}(R)
}
$$
in which all the squares are Cartesian -- for the lower square, this follows from Construction \ref{Cons untwisted TC^-}. In particular, we may compute $\mathrm{TC}^{(-1)}(R)$ as the pullback if the total square of the above diagram. On the other hand, the factorization \eqref{Factoring the Tate Frobenius} of the Nikolaus-Scholze Tate-Frobenius $\varphi_\mathrm{NS}:\mathrm{TC}^-(R)^\wedge_p\to\mathrm{TP}(R)^\wedge_p$ through $\varphi^{(-1)}$  from the proof of Theorem \ref{Theorem Frobenius untwist for TP} allows us to write a commutative diagram of pullback squares 
$$
\xymatrix{
\mathrm{TC}(R)^\wedge_p\ar[rr] \ar[d]^{} & & \mathrm{TC}^-(R)^\wedge_p\ar[d]^{(\mathrm{can},\, \varphi^{(-1)})} \\ \mathrm{TP}^{(-1)}(R)\ar[rr]^{\Delta\qquad} \ar[d]^{} & & \mathrm{TP}(R)^\wedge_p\times \mathrm{TP}^{(-1)}(R) \ar[d]^{}\\
\mathrm{TP}(R){}^\wedge_p\ar[rr]^{\Delta\qquad} & & \mathrm{TP}(R)^\wedge_p\times \mathrm{TP}(R)^\wedge_p,
}
$$
in which the identification of the upper left vertex with $p$-completed topological cyclic homology stems from noting that the total square exhibits said vertex as an equalizer, and hence as the fiber
$$
\mathrm{fib}\big(\mathrm{TC}^{-}(R)^\wedge_p\xrightarrow{\varphi_\mathrm{NS}\,-\,\mathrm{can}}\mathrm{TP}(R)^\wedge_p\big)\, \simeq \, \mathrm{TC}(R)^\wedge_p.
$$
Since the total square of the first commutative diagram is the upper square of the second diagram, and they all pullbacks, we obtained the desired identification of the total pullbacks $\mathrm{TC}^{(-1)}(R)\simeq \mathrm{TC}(R)^\wedge_p$.
\end{proof}

\begin{remark}
By the passage to homotopy groups and evenness, the assertion of Proposition \ref{Prop TC don't care} is equivalent to the statement that the syntomic complexes $\mathbf Z_p(i)(R)$ are insensitive to Nygaard-completion for all $i\in \mathbf Z$, in the sense that
$$
\mathbf Z_p(i)(R)\, :=\, \mathrm{fib}\big(\mathrm{Fil}^i_{\mathcal N}\,\mathbbl \Delta_R\{i\}\xrightarrow{\varphi\,-\,1}\mathbbl\Delta_R\{i\}\big)\, \simeq \,
 \mathrm{fib}\big(\mathrm{Fil}^i_{\mathcal N}\,\widehat{\mathbbl \Delta}_R\{i\}\xrightarrow{\widehat{\varphi}\,-\,1}\widehat{\mathbbl\Delta}_R\{i\}\big)
$$
This can be shown directly via the same argument as we have given above, see \cite[Proposition 7.12]{Antieau-Krause-Nikolaus}.
\end{remark}

\begin{remark}
The constructions of this section are expected to be related to the decompleted variants of topological cyclic homology $\mathrm{TP}_{\mathbbl\Delta}(R)$, $\mathrm{TC}^-_{\mathbbl\Delta}(R)$ from the upcoming work of Devalapurkar-Hahn-Raksit-Yuan. Part of their construction is discussed in
\cite[Construction 5.2]{Equivariant aspects of de-completing cyclichomology}, where a genuine $\mathbf T$-equivariant $\mathrm{THH}_{\mathbbl\Delta}(R)$ is constructed for any animated ring $R$, not merely qrsp.
\end{remark}

The Frobenius-untwisted functors $\mathrm{TP}^{(-1)}$ and $\mathrm{TC}^{-(-1)}$ were defined above only for qrsp ring, but they can be extended to a broader class of inputs via quasi-syntomic descent.

\begin{cons}\label{TC and TP for QSyn}
Let $\mathrm{FSch}^\mathrm{qsyn}_{\Z_p}$ denote the category of quasi-syntomic $p$-adic formal scheme. The restriction of their  functors of points
 to the full subcategory of qrsp affines induces a functor
$
\mathrm{FSch}^\mathrm{qsyn}_{\Z_p}\hookrightarrow\mathrm{Shv}_{\mathrm{qsyn}}(\CAlg^\mathrm{qrsp}_{\Z_p}),
$
which is fully faithful by \cite[Proposition 4.31]{BMS2}. Since the $\i$-category of $\E$-rings $\CAlg$ is has all small limits, the
 functors $\mathrm{TP}^{(-1)}, \mathrm{TC}^{-(-1)}:\CAlg_{\Z_p}^\mathrm{qrsp}\to\CAlg$
give rise via right Kan extension to limit-preserving functors $\Fun(\CAlg_{\Z_p}^\mathrm{qrsp}, \mathrm{Ani})^\mathrm{op}\to \CAlg$.
We now define the functors $X\mapsto \mathrm{TP}^{(-1)}(X)$ and $X\mapsto \mathrm{TC}^{-(-1)}(X)$ for quasi-syntomic $p$-adic formal schemes $X$ as the composites
$$
(\mathrm{FSch}^\mathrm{qsyn}_{\Z_p})^\mathrm{op}\hookrightarrow\mathrm{Shv}_{\mathrm{qsyn}}(\CAlg^\mathrm{qrsp}_{\Z_p})^\mathrm{op}\subseteq\Fun(\CAlg^\mathrm{qrsp}_{\Z_p}, \mathrm{Ani})^\mathrm{op}\to \CAlg.
$$
\end{cons}

\begin{remark}\label{Colimit presentation of TP and TC}
 The Frobenius-untwisted topological periodic and negative cyclic homology of a quasi-syntomic $p$-adic formal scheme $X$ is given by
$$
\mathrm{TP}^{(-1)}(X) \, \simeq \varprojlim_{R\in \mathrm{CAlg}^\mathrm{qrsp}_{\mathbf Z_p/X}} \mathrm{TP}^{(-1)}(R), \qquad
\text{\textlira} X \,\simeq \varprojlim_{R\in \mathrm{CAlg}^\mathrm{qrsp}_{\mathbf Z_p/X}} \mathrm{TC}^{-(-1)}(R),
$$
where the indexing category $\mathrm{CAlg}^\mathrm{qrsp}_{\Z_p/X}$ consists of a pair of a qrsp ring $R$ together with a map $\Spf(R)\to X$.
\end{remark}

By the construction of the natural transformation $\varphi, \mathrm{can}:\mathrm{TP}^{(-1)}\to\mathrm{TC}^{-(-1)}$ given above between functors $\CAlg_{\Z_p}^\mathrm{qrsp}\to\CAlg$, it is clear that they also extend to natural transformations of functors $(\mathrm{FSch}^\mathrm{qsyn}_{\Z_p})^\mathrm{op}\to\CAlg$.
At this point, we obtain a version of
Proposition \ref{Prop TC don't care}
for quasi-syntomic $p$-adic formal schemes for no extra work.

\begin{prop}\label{Prop TC don't care II}
For any quasi-syntomic $p$-adic formal scheme $X$, there is a canonical equivalence of $\E$-rings
$$
\mathrm{TC}(X)^\wedge_p\,\simeq \, \mathrm{fib}\big(\mathrm{TC}^{-(-1)}(X)\xrightarrow{\varphi\,-\,\mathrm{can}}\mathrm{TP}^{(-1)}(X)\big).
$$
\end{prop}

\begin{proof}
Since all of  the functors $\mathrm{TC}^\wedge_p$, $\mathrm{TC}^{-(-1)},$ and $\mathrm{TP}^{(-1)}$ are  extended from qrsp rings in a limit-preserving way, and fibers also commute with limits, this follows directly from Proposition \ref{Prop TC don't care}.
\end{proof}

\section{The even periodic enhancements}

After having defined the Frobenius-untwisted topological cyclic homology variants $\mathrm{TP}^{(-1)}$ and $\mathrm{TC}^{-(-1)}$ in the previous section, we now wish to extend them via descent to functors that intake and both output algebro-geometric objects. On the most basic level, we roughly wish to pass from the associations
$$
\Spf(R)\mapsto \Spf\left(\mathrm{TP}^{(-1)}(R)\right), \qquad \Spf(R)\mapsto \Spf\left(\mathrm{TC}^{-(-1)}(R)\right)
$$
for qrsp rings $R$ by means of descent to certain functors
$$
\{\text{$p$-adic formal schemes}\} \to \{\text{formal spectral stacks}\},
$$
with appropriate restrictions on both sides and corresponding modifications. We will accomplish this in Construction \ref{Const of pounds}.

For this purpose, we require a theory of formal spectral algebraic geometry. If we could restrict attention to the connective setting -- i.e.\ spectral algebraic geometry built upon formal affines $\Spf(A)$ where $A$ are connective adic $\E$-rings-- then \cite[Chapter 8]{SAG} would be entirely satisfactory. 
But since we are seeking even \textit{periodic} enhancements of the prismatization stacks, the connective assumption would be untenable.
Instead, we adopt a  functor of points approach to  formal spectral stacks. We briefly summarize the setup, but direct the reader to consult \cite[Section 5]{Even periodization} for details.

 \begin{summary}
 The $\i$-category of \textit{formal spectral stacks} is taken to be the full subcategory
 $$
 \mathrm{FSpStk}\,:=\,\mathrm{Shv}_\mathrm{fpqc}^\mathrm{acc}(\CAlg_\mathrm{ad})\,\subseteq\,\Fun(\CAlg_\mathrm{ad}, \mathrm{Ani})
 $$
 spanned by all the accessible sheaves for an adic version of the fpqc topology.
The variants of \textit{even periodic formal spectral stacks} and \textit{classical formal spectral stacks} are respectively obtained as
 $$
 \mathrm{FSpStk}^\mathrm{evp}\, :=\, \mathrm{Shv}_\mathrm{fpqc}^\mathrm{acc}(\CAlg^\mathrm{evp}_\mathrm{ad}),\qquad
 \mathrm{FStk}^\heart\,:=\, \mathrm{Shv}_\mathrm{fpqc}^\mathrm{acc}(\CAlg^\heart_\mathrm{ad}),
 $$
defined in terms of the subcategories $\CAlg_\mathrm{ad}^\mathrm{evp}, \CAlg_\mathrm{ad}^\heart\subseteq\CAlg_\mathrm{ad}$ of \textit{even periodic} and \textit{classical adic $\E$-rings} respectively. There are canonical functors
$$
(-)^\mathrm{evp} : \mathrm{FSpStk}\to\mathrm{FSpStk}^\mathrm{evp}, \qquad (-)^\heart : \mathrm{FSpStk}\to \mathrm{FStk}^\heart,
$$
the \textit{even periodization} and \textit{underlying classical stack},  informally given by
$$
\mathfrak X^\mathrm{evp}\,\,\simeq \varinjlim_{A\in\CAlg^\mathrm{ev}_{\mathrm{ad}/\mathfrak X}}\Spf(A), \qquad \mathfrak X^\heart\,\,\simeq \varinjlim_{A\in\CAlg^\heart_{\mathrm{ad}/\mathfrak X}}\Spf(\pi_0(A)).
$$
The functors of \textit{the $\E$-ring of global sections} and \textit{$\i$-category of quasi-coherent sheaves}
$$
\sO : \mathrm{FSpStk}^\mathrm{op}\to \CAlg, \qquad \QCoh :\mathrm{FSpStk}^\mathrm{op}\to \CAlg(\mathcal P\mathrm{r^L})
$$
are obtained by right Kan extension and fpqc sheafification of the respective functors
\begin{equation}\label{O and QCoh wish}
\Spf(A)\mapsto A, \qquad \Spf(A)\to\Mod_A^\mathrm{cplt}.
\end{equation}
When restricted to $\mathrm{FSpStk}^\mathrm{evp}$ and $\mathrm{FStk}^\heart$ (as well as the obviously-defined $\mathrm{FSpStk}^\mathrm{cn}$), the functors $\sO$ and $\QCoh$ do not require fpqc sheafification, and therefore \textit{recover} the functors \eqref{O and QCoh wish} on formal affines.
 \end{summary}

The even periodization functor $\mathfrak X\mapsto \mathfrak X^\mathrm{evp}$ provides  the universal means of replacing a not-necessarily-even-periodic spectral stack with an even periodic one. This is the role which it will play in this paper too, but this role will be rather inessential thanks to the following explicit description:

\begin{lemma}\label{Lemma evp for TC^-}
Let $R$ be a qrsp ring. The even periodization. There is a canonical and natural equivalence of even periodic formal spectral stacks
$$
\Spf\left(\mathrm{TC}^{-(-1)}(R)\right){}^\mathrm{evp}\,\simeq \, \Spf\left(\mathrm{TC}^{-(-1)}(R)\otimes_{\mathrm{MU}}\mathrm{MUP} \right)/\mathbf G_m.
$$
\end{lemma}

\begin{proof}
Observe that $\mathrm{TC}^{-(-1)}(R)$ is by construction an even $\E$-algebra over $\mathbf Z$ by Proposition \ref{Prop homotopy groups of untiwsted TC^-}. It is therefore also an even $\E$-algebra over $\mathrm{MU}$ via the $\E$-ring map $\mathrm{MU}\to \pi_0(\mathrm{MU})\simeq\mathbf Z$ also an even $\E$-algebra over $\mathrm{MU}$. The desired equivalence is thus a special case of \cite[Proposition 4.5.4]{Even periodization}, according to which the even periodization for any even adic $\E$-algebra $A$ over $\mathrm{MU}$ is given by
$$
\Spf(A)^\mathrm{evp}\, \simeq\,\Spf\left(A\otimes_{\mathrm{MU}}\mathrm{MUP} \right)/\mathbf G_m,
$$
where the $\mathbf G_m$-action quotiented out on the right corresponds in the usual way to the  $\Z$-grading exhibited by $A\otimes_{\mathrm{MU}}\mathrm{MUP}\simeq \bigoplus_{i\in \mathbf Z}\Sigma^{-2i}(A)$.
\end{proof}

Since we wish to obtain even periodic spectral stacks, Lemma \ref{Lemma evp for TC^-} suggests how to modify $\mathrm{TC}^{-(-1)}$. The functor $\mathrm{TP}^{(-1)}$ requires no such modification, as their values on qrsp rings are even periodic by Theorem \ref{Theorem Frobenius untwist for TP}. To extend these functors to stacks as we desire, we must verify that they are appropriately compatible with the Grothendieck topologies on both sides.

\begin{prop}\label{Prop Descent for QRSPs}
The functors $\mathrm{CAlg}_{\mathbf Z_p}^\mathrm{qrsp}\to\mathrm{CAlg}^\mathrm{evp}_\mathrm{ad}$, given by
$$
R\mapsto\mathrm{TP}^{(-1)}(R),\qquad  R\mapsto \mathrm{TC}^{-(-1)}(R)\otimes_{\mathrm{MU}}\mathrm{MUP}
$$
take quasi-syntomic covers to adically faithfully flat covers.
\end{prop}

\begin{proof}
Let $R\to S$ be a quasi-syntomic cover of qrsp rings. We must show that the corresponding maps of adic $\E$-ring
$$
\mathrm{TP}^{(-1)}(R)\to \mathrm{TP}^{(-1)}(S), \qquad \mathrm{TC}^{-(-1)}(R)\o_{\mathrm{MU}}\mathrm{MUP}\to \mathrm{TC}^{-(-1)}(S)\otimes_{\mathrm{MU}}\mathrm{MUP}
$$
are both $(p, I)$-adically faithfully flat, where 
$$
I\subseteq \mathbbl\Delta_R =\pi_0\left(\mathrm{TP}^{(-1)}(R)\right) = \pi_0\left(\mathrm{TC}^{-(-1)}(R)\right)
$$
denotes the prismatic ideal.
Seeing how the functors $\mathrm{TP}^{(-1)}$ and 
$$
\mathrm{TC}^{-(-1)}\otimes_{\mathrm{MU}}\mathrm{MUP}\simeq \bigoplus_{i\in \mathbf Z} \Sigma^{-2i} \mathrm{TC}^{-(-1)}
$$
both output $p$-torsion-bounded even periodic $\E$-rings, it suffices to show that the corresponding statement holds on $\pi_{0}$. That is to say, the result in question boils down to $\mathbbl\Delta_R\to\mathbbl\Delta_S$, as well as their respective map between the Rees constructions of their Nygaard filtrations, are $(p, I)$-completely faithfully flat.

For $\mathbbl\Delta_R$, this follows by combining that by \cite[Lemma 6.3]{BLb},   prismatization  $R\mapsto\Spf(R)^{\mathbbl\Delta}$ sends quasi-syntomic covers to faithfully flat covers,  and the formal affineness result $\mathrm{Spf}(R)^{\mathbbl\Delta}\simeq\Spf(\mathbbl\Delta_R)$ of \cite[Theorem 5.5.7]{Bhatt F-gauges}  for any qrsp ring $R$.
For the Nygaard filtration $\mathrm{Fil}_{\mathcal N}^*\,\mathbbl\Delta_R$, or more precisely for its Rees algebra, we can similarly deduce this from the identification of the filtered prismatization
$$
\mathrm{Spf}(R)^{\mathcal N}\simeq \Spf\big(\bigoplus_{i\in \mathbf Z}\mathrm{Fil}^i_{\mathcal N}\,\mathbbl\Delta_R \big )/\mathbf G_m
$$
of a qrsp ring $R$
from \cite[Theorem 5.5.10]{Bhatt F-gauges}.
\end{proof}

At this point, we have everything needed to construct the desired extensions of $\mathrm{TP}^{(-1)}$ and $\mathrm{TC}^{-(-1)}$ to stacks. These  will also turn out by Theorem \ref{Main Theorem in Text} to be the titular even periodic enhancements of prismatization stacks that this paper is about.

\begin{cons}\label{Const of pounds}
The functors $\mathrm{CAlg}^\mathrm{qrsp}_{\mathbf Z_p}\to \mathrm{FSpStk}^\mathrm{evp}$, given by
$$
R\mapsto \Spf\left(\mathrm{TP}^{(-1)}(R)\right), 
\qquad R
\mapsto\Spf\left(\mathrm{TC}^{-(-1)}(R)\right){}^\mathrm{evp},
$$
extend by Proposition \ref{Prop Descent for QRSPs} to the respective $\i$-categories of accessible sheaves
$$
\pounds\,:\,\mathrm{Shv}^\mathrm{acc}_\mathrm{qsyn}(\mathrm{CAlg}_{\mathbf Z_p}^\mathrm{qrsp})\to \mathrm{FSpStk}^\mathrm{evp}, \quad
\text{\textlira}\,:\,\mathrm{Shv}^\mathrm{acc}_\mathrm{qsyn}(\mathrm{CAlg}_{\mathbf Z_p}^\mathrm{qrsp})\to \mathrm{FSpStk}^\mathrm{evp}.
$$
We use the same notation $X\mapsto \pounds X, \text{\textlira}X$ to denote the composites
$$
\mathrm{FSch}_{\mathbf Z_p}^\mathrm{qsyn}\hookrightarrow\mathrm{Shv}^\mathrm{acc}_\mathrm{qsyn}(\mathrm{CAlg}_{\mathbf Z_p}^\mathrm{qrsp})\xrightarrow{\pounds, \text{\textlira}}\mathrm{FSpStk}^\mathrm{evp}
$$
of the above-constructed  functors $\pounds, \text{\textlira}$
with the restriction of the functors of points of quasi-syntomic $p$-adic formal schemes to the full subcategory of qrsp affines.
\end{cons}

We refer to $\pounds X$ as the \textit{spherical Tate-loop space} and to $\text{\textlira}X$ as the \textit{spherical Nygaard-loop space} of the quasi-syntomic $p$-adic formal scheme $X$ respectively.

\begin{remark}\label{Colimit presentation of pounds and lira}
Let $X$ be
a quasi-syntomic $p$-adic formal scheme. Its corresponding spherical Tate-loop and Nygaard-loop spaces may be expressed as
$$
\pounds X \, \simeq\, \varinjlim_{R\in \mathrm{CAlg}^\mathrm{qrsp}_{\mathbf Z_p/X}} \Spf\left(\mathrm{TP}^{(-1)}(R)\right), \quad
\text{\textlira} X \,\simeq \varinjlim_{R\in \mathrm{CAlg}^\mathrm{qrsp}_{\mathbf Z_p/X}} \Spf\left(\mathrm{TC}^{-(-1)}(R)\right){}^\mathrm{evp},
$$
where the indexing category $\mathrm{CAlg}^\mathrm{qrsp}_{\Z_p/X}$ consists of pairs of a qrsp rings $R$ together with a map $\Spf(R)\to X$.
\end{remark}

\begin{remark}
In colimit formulas, such as those for $\pounds X$, and $\text{\textlira}X,$ in  Remark \ref{Colimit presentation of pounds and lira}, it is sometimes convenient to replace the indexing category $\mathrm{CAlg}^\mathrm{qrsp}_{\mathbf Z_p/X}$ with a cofinal small subcategory. For instance, we may fix a formally affine Zariski open cover $\{U_i\}$ of $X$, and take $\mathcal C\subseteq\mathrm{CAlg}^\mathrm{qrsp}_{/X}$ to be the full subcategory of quasi-syntomic morphisms of the form $\Spf(R)\to U_i\subseteq X$ for any $i$. In light of such a substitution being straightforward, we will not explicitly comment upon it.
\end{remark}

\begin{remark}\label{Remark spherical loops og}
Both the terminology of the spherical Tate-loop and Nygaard-loop spaces, as well as the notation $\pounds, \text{\textlira}$, are all intended to encourage viewing  the spectral stacks $\pounds X$  and $\text{\textlira}X$ as modifications of
$$
\mathcal LX\,:=\, (\mathscr LX)^{\mathrm{evp}}\,\,\simeq  \varinjlim_{R\in \mathrm{CAlg}^\mathrm{qrsp}_{\mathbf Z_p/X}}\mathrm{Spf}\left(\mathrm{THH}(R)^\wedge_p\right){}^\mathrm{evp},
$$
the even periodization of the \textit{spherical free loop space} $\mathscr LX \simeq\underline{\Map}_{\mathrm{SpStk}}(\underline{\mathbf T}_{\Spec(\mathbf S)}, X)$ on $X$. The adjective \textit{spherical} is intended to emphasize that, even though $X$ is defined over $\Spf(\mathbf Z_p)$, we are taking the loop space over the sphere spectrum, or more precisely, over $\mathrm{Spf}(\mathbf S^\wedge_p)$.
\end{remark}

The Tate-loop and Nygaard-loop spaces $\pounds X$ and $\text{\textlira}X$ of Construction \ref{Const of pounds}
are the quasi-syntomic-descended stacky versions of $\mathrm{TP}^{(-1)}$ and $\mathrm{TC}^{-(-1)}$ respectively.
The next construction is roughly an analogous stacky variant of topological cyclic homology $\mathrm{TC}$ (which is invariant under Frobenius untwisting by Proposition \ref{Prop TC don't care}).

\begin{cons}
The maps of $\E$-rings 
$$
\mathrm{can}, \varphi :\mathrm{TC}^{-(-1)}(R)\rightrightarrows \mathrm{TP}^{(-1)}(R)
$$
for qrsp rings $R$ from Constructions \ref{Cons untwisted TC^-} and \ref{Cons of div Frob} induce through the mechanism of Construction \ref{Const of pounds} a canonical pair of eponymous maps of even periodic formal spectral stacks from $\text{\textlira}X$ to $\pounds X$. Let the even periodic formal spectral stack $\$ X$ be their coequalizer. That is to say, we set
$$
\$ X\,:= \,\mathrm{coeq}\big(\xymatrix{ \pounds X \ar@<0.6ex>[r]^\varphi \ar@<-0.6ex>[r]_{\mathrm{can}} & \text{\textlira}X}\big),
$$
and refer to it as the \textit{spherical syntomification of $X$}.
\end{cons}

\begin{remark}
In the usual way of rephrasing an coequalizer as a pushout, the spherical syntomification of $X$ may be equivalently obtained through the pushout square
$$
\begin{tikzcd}
\pounds X \coprod \pounds X\arrow{d}{\nabla} \arrow{rr}{(\mathrm{can},\,\varphi)\,\,\,} & & \text{\textlira} X\arrow{d}{}\\
\pounds X\arrow{rr}{} & & \$X.
\end{tikzcd}
$$
Comparing this with the definition of the syntomification stack as the pushout
$$
\begin{tikzcd}
 X^{\mathbbl\Delta} \coprod  X^{\mathbbl\Delta}\arrow{d}{\nabla} \arrow{rr}{(j_\mathrm{dR},\,j_\mathrm{HT})\,\,\,} & & X^{\mathcal N}\arrow{d}{}\\
X^{\mathbbl\Delta}\arrow{rr}{} & & X^\mathrm{Syn}
\end{tikzcd}
$$
from \cite[Remark 5.5.18]{Bhatt F-gauges} 
suggests a relationship between syntomification and spherical syntomification which is borne out by the next result.
\end{remark}

\begin{theorem}\label{Main Theorem in Text}
Let $X$ be a  quasi-syntomic $p$-adic formal scheme $X$. There are canonical natural equivalences of classical stacks
$$
(\begin{normalfont}\pounds \end{normalfont}X)^\heartsuit \simeq X^{\mathbbl\Delta},\qquad  (\begin{normalfont}\text{\textlira}\end{normalfont}X)^\heart\simeq X^\mathcal N,\qquad (\begin{normalfont}\$\end{normalfont}X)^\heart\simeq X^\mathrm{Syn}.$$
\end{theorem}

\begin{proof}
Recall that the passage to the underlying classical stack $\mathfrak X\mapsto\mathfrak X^\heart$ commutes with all small colimits when viewed as a functor
$\mathrm{FSpStk}^\mathrm{evp}\to \mathrm{FStk}^\heart\simeq \mathrm{Stk}^\heart$.
In light of the colimit formula for $\pounds X$ and from Presentation \ref{Colimit presentation of pounds and lira}, we find the underlying classical stack of the spherical Tate-loop space of $X$  expressed as
$$
(\pounds X)^\heart \,\, \simeq \varinjlim_{R\in \mathrm{CAlg}^\mathrm{qrsp}_{\mathbf Z_p/X}} \Spf\left(\pi_0\left(\mathrm{TP}^{(-1)}(R)\right)\right).
$$
By the natural identification $\pi_0\left(\mathrm{TP}^{(-1)}(R)\right)\simeq \mathbbl\Delta_R$ from Theorem \ref{Theorem Frobenius untwist for TP}, we  find that
$$
(\pounds X)^\heart \, \simeq \varinjlim_{R\in \mathrm{CAlg}^\mathrm{qrsp}_{\mathbf Z_p/X}} \Spf(\mathbbl\Delta_R)\,\simeq \, X^{\mathbbl \Delta},
$$
where the identification with the prismatization of $X$ is obtained by combining \cite[Theorem 5.5.7]{Bhatt F-gauges} and \cite[Lemma 6.3]{BLb}.

We turn towards an analogous identification of the underlying classical stack $(\text{\textlira}X)^\heart$ of the spherical Nygaard-loop space. By the same argument as before, it may be write it in the form
$$
(\text{\textlira} X)^\heart
\simeq \varinjlim_{R\in \mathrm{CAlg}^\mathrm{qrsp}_{\mathbf Z_p/X}} \left(\Spf\left(\mathrm{TC}^{-(-1)}(R)\right){}^\mathrm{evp}\right){}^\heart.
$$
For an arbitrary fixed qrsp ring $R$,
the relevant underlying classical stack may be expressed according to Lemma \ref{Lemma evp for TC^-} as
\begin{eqnarray*}
\left(
\Spf\left(\mathrm{TC}^{-(-1)}(R)\right){}^\mathrm{evp}\right){}^\heart &\simeq & \Spf\Big(\bigoplus_{i\in \mathbf Z} \pi_{2i}\big(\mathrm{TC}^{-(-1)}(R)\big)  \Big)/\mathbf G_m \\
&\simeq & \Spf\Big(\bigoplus_{i\in \mathbf Z} \mathrm{Fil}^i_{\mathcal N}\, \mathbbl\Delta_R\{i\}  \Big)/\mathbf G_m\\
&\simeq& \Spf(R)^{\mathcal N}.
\end{eqnarray*}
Noting that all the above identifications are natural in the variable $R\in\mathrm{CAlg}_{\mathbf Z_p}^\mathrm{qrsp}$, the equivalence for the filtered prismatization
$$
X^\mathcal N \,\, \simeq \varinjlim_{R\in \mathrm{CAlg}^\mathrm{qrsp}_{\mathbf Z_p/X}}\Spf(R)^\mathcal N
$$
from \cite[Remark 5.5.18]{Bhatt F-gauges} completes the desired computation of $(\text{\textlira}X)^\heart$. 

Finally, we turn to the spherical syntomification $\$X$. Recall that the passage to the underlying classical stacks $\mathfrak X\mapsto \mathfrak X^\heart$ commutes with small colimits and as such in particular also with coequalizers. The desired identification of $(\$X)^\heart\simeq X^\mathrm{Syn}$ therefore follows from what we have already shown together with the definition of the syntomification in \cite[Definition 6.1.1]{Bhatt F-gauges} as the coequalizer
  $$
X^\mathrm{Syn}\,\simeq \,\mathrm{coeq}\big(\xymatrix{ X^{\mathbbl \Delta} \ar@<0.6ex>[r]^{j_\mathrm{HT}} \ar@<-0.6ex>[r]_{j_\mathrm{dR}} & X^{\mathcal N}}\big),
 $$
 so long as we show that $\varphi^\heart\simeq j_\mathrm{HT}$ and $\mathrm{can}^\heart\simeq j_\mathrm{dR}$. In light of the fact that $j_\mathrm{HT}$ and $j_\mathrm{dR}$ are quasi-syntomically-descended from qrsp affines by \cite[Remark 5.5.5]{Bhatt F-gauges}, this follows from the construction of Frobenius-untwisted versions of the maps $\varphi$ and $\mathrm{can}$, specifically from Proposition \ref{Prop Frob on homotopy}. 
\end{proof}

\begin{remark}
In \cite[Section 6]{Even periodization}, we define a \textit{coarse quotient} of an even periodic formal spectral stack $\mathfrak X$ with respect to a $\T$-action. In the case of even periodization $\mathcal LX$ of the spherical free loop spaces from Remark \ref{Remark spherical loops og} on a quasi-syntomic $p$-adic formal scheme $X$ and its loop-rotation $\T$-action, it follows from \cite[Proof of Theorem 6.4.5]{Even periodization} that the coarse quotient is given by
\begin{equation}\label{Colimit presentation for LX//T}
\mathcal LX/\!/\T\, \simeq \,  \varinjlim_{R\in \mathrm{CAlg}^\mathrm{qrsp}_{\mathbf Z_p/X}}\mathrm{Spf}\left(\mathrm{TC}^-(R)^\wedge_p\right){}^\mathrm{evp}.
\end{equation}
We saw in \cite[Theorem 6.4.5]{Even periodization} that the underlying classical stack of this spectral stack recovers the Nygaard-completed filtered prismatization $X^{\mathcal N}$, i.e.\ the pullback of the Cartesian square of classical stacks
$$
\begin{tikzcd}
X^{\widehat{\mathcal N}} \arrow{r}{} \arrow{d}{} & X^{\mathcal N}\arrow{d}{}\\
\widehat{\mathbf A}^1/\mathbf G_m \arrow{r}{} & \mathbf A^1/\mathbf G_m.
\end{tikzcd}
$$
We therefore have two different even periodic spectral enhancements $\mathcal LX/\!/\T$ and $\text{\textlira X}$ of (versions of) the filtered prismatization $X^{\widehat{\mathcal N}}$ and $X^{\mathcal N}$ respectively. 
Using the synchronized colimit presentations of Remark \ref{Colimit presentation of pounds and lira} and \eqref{Colimit presentation for LX//T}, we see that the canonical natural map $\mathrm{TC}^-(R)^\wedge_p\to \mathrm{TC}^{-(-1)}(R)$ for $R\in\CAlg_{\Z_p}^\mathrm{qrsp}$ induces a natural map of even periodic formal spectral stacks
$$
u: \mathcal LX/\!/\T \to \text{\textlira}X,
$$
from which we may recovers the upper horizontal arrow in the above pullback square by passing to the underlying classical stacks. The factorization of the Frobenius through the non-Frobenius-untwisted negative topological cyclic homology from the proof of Proposition \ref{Prop Frob on homotopy} globalizes to even periodic spectral stacks and gives rise to a commutative diagram
$$
\begin{tikzcd}
\pounds X \arrow{d}{} \arrow{r}{}\arrow[bend left=20]{rr}{\varphi}&  \mathcal LX/\!/\T 
 \arrow{r}{u} \arrow{d}{} & \text{\textlira} X\arrow{d}{}  \\
\pounds \Spf(\Z_p) \arrow{r}{}\arrow[bend right=20]{rr}{\varphi}&  \mathcal L\Spf(\Z_p)/\!/\T 
 \arrow{r}{u}  & \text{\textlira} \Spf(\Z_p)  \\\end{tikzcd},
$$
whose inner left square is Cartesian (this follows from the construction of $\mathrm{TP}^{(-1)}$ as a localization of $\mathrm{TC}^-$). That is to say, it suffices to perform the modification, required of the coarse-rotation-quotiented even-periodized spherical loop space $\mathcal LX/\!/\T$ to produce the spherical Tate-loop space $\pounds X$, in the base-case $X=\Spf(\Z_p)$, and then extend it via pullback to an arbitrary quasi-syntomic $p$-adic formal scheme $X$.
\end{remark}

\begin{remark}
Let us describe the spherical Tate-loop space via transmutation.
The functor of points of $\pounds X :\CAlg_\mathrm{ad}^\mathrm{evp}\to\mathrm{Ani}$ may be written for any even periodic adic $\E$-ring $A$ as
$$
\mathrm{Map}_{\mathrm{FSpStk}^\mathrm{evp}}(\mathrm{Spf}(A),\pounds X) \,\, \simeq  \varinjlim_{\substack{\mathrm{TP}^{(-1)}(R)\to A\\ R\in\mathrm{CAlg}^\mathrm{qrsp}_{\mathbf Z_p}}} X(R).
$$
We suggest this might be viewed as in analogy to the description of the functor of points of the prismatization $X^{\mathbbl\Delta}$ \textit{as a $\delta$-stack}. Indeed,  for any $\delta$-ring $A$ we have
$$
\Map_{\delta\mathrm{Stk}^\heart}(\mathrm{Spec}(A), X^{\mathbbl\Delta})
 \,\, \simeq  \varinjlim_{\overline A\in \mathrm{Prism}(A)} X(\overline A),
$$
where $\mathrm{Prism}(A)$ is the collection of animated prism structures on $A$: generalized Cartier divisors $A\to \overline A$, satisfying some conditions\footnote{Applying the forgetful functor $\delta\mathrm{Stk}^\heart\to\mathrm{Stk}^\heart$ to the $\delta$-stack $X^{\mathbbl\Delta}$ corresponds to pre-composing with the right adjoint to the forgetful functor $\delta\mathrm{CAlg}\rightleftarrows \mathrm{CAlg} : \mathrm W$, given by Witt vectors. Since animated prism structures on $\mathrm W(R)$ are precisely Cartier-Witt divisors, this recovers the perhaps more familiar formula for $X^{\mathbbl\Delta}$}. For some connection between $\E$-rings and $\delta$-rings, specifically using the Nikolaus-Scholze Tate-Frobenius map $\varphi_\mathrm{NS} :A\to A^{\mathrm {tC}_p}$, see \cite{Nikolaus Group rings}. This suggests that a good notion of  a \textit{spectral prism} might allow the canonical map $\mathrm{TP}(R)^\wedge_p\to \mathrm{THH}(R)^{\mathrm{tC}_p}$ for $R\in\mathrm{CAlg}^\mathrm{qrsp}_{\mathbf Z_p}$ to be an instance of a spectral prism structure $A\to\overline A$, at least after appropriate Frobenius untwisting (e.g.\ perhaps a base-change along $\mathrm{TP}(R)^\wedge_p\to\mathrm{TP}^{(-1)}(R)$).
\end{remark}

In Theorem \ref{Main Theorem in Text} we have identified the underlying classical stacks of the even periodic stacks constructed in this section. 
It is rather simple to also identify
 their $\E$-rings of global functions in terms of the Frobenius-untwists of topological cyclic homology for quasi-syntomic $p$-adic formal schemes from Construction \ref{TC and TP for QSyn}.

\begin{prop}
Let $X$ be a  quasi-syntomic $p$-adic formal scheme $X$. There are canonical natural equivalences of $\E$-rings
$$
\sO(\begin{normalfont}\pounds \end{normalfont}X)\, \simeq\, \mathrm{TP}^{(-1)}(X),
\quad
\sO(\begin{normalfont}\text{\textlira}\end{normalfont}X) \, \simeq \, \mathrm{TC}^{-(-1)}(X),
\quad
\sO(\begin{normalfont}\$\end{normalfont}X) \,\simeq\, \mathrm{TC}(X)^\wedge_p.
$$
\end{prop}

\begin{proof}
Since quasi-syntomic $p$-adic formal schemes are generated under colimits by qrsp affines, and all the functors in questions take all small colimits in $X$ to the corresponding limits of $\E$-rings, it suffices to restrict to the case $X=\Spf(R)$ for a qrsp ring $R$. In that case. the spherical Tate-loop and Nygaard-loop spaces taken on the simple form
$$
\pounds \Spf(R)^{\mathbbl\Delta}\,\simeq\, \Spf\left(\mathrm{TP}^{(-1)}(R)\right),
\quad
\pounds \Spf(R)^{\mathcal N}\,\simeq\, \Spf\left(\mathrm{TC}^{-(-1)}(R)\otimes_{\mathrm{MU}}\mathrm{MUP}\right)/\mathbf G_m,
$$
from which it is immediate that $\sO\left(
\pounds \Spf(R)^{\mathbbl\Delta}
\right)\simeq \mathrm{TP}^{(-1)}(R)$, and since $\mathbf G_m$-fixed points correspond in terms of gradings to taking the $0$-th graded piece, 
we also get
$$
\sO\left(\Spf(R)^{\mathcal N}\right)
\,\simeq \,\left(\mathrm{TC}^{-(-1)}(R)\otimes_{\mathrm{MU}}\mathrm{MUP}\right)^{\G_m}\,\simeq \, \mathrm{TC}^{-(-1)}(R).
$$
For the spherical syntomification, we find an equivalence of $\E$-rings
$$
\sO\left(\Spf(R)^{\mathrm{Syn}}\right)\,\simeq\,\mathrm{eq}
\big(\xymatrix{ \mathrm{TC}^{(-1)}(X) \ar@<0.6ex>[r]^{\varphi} \ar@<-0.6ex>[r]_{\mathrm{can}} & \mathrm{TP}^{(-1)}(X)}\big)\,\simeq\, \mathrm{TC}(X)^\wedge_p,
$$
where the second equivalence comes from Proposition \ref{Prop TC don't care II}.
\end{proof}

\section{Recovering structure on prismatization}
Theorem \ref{Main Theorem in Text} shows that the even periodic stacks $\pounds X$, $\text{\textlira}X$, and $\$X$ are  spectral enhancements of the  prismatization stacks $X^{\mathbbl\Delta}$, $X^{\mathcal N}$, and $X^{\mathrm{Syn}}$ for any quasi-syntomic $p$-adic formal scheme $X$. We now show how various structure on the prismatization stacks may be viewed as arising from these even periodic enhancements. 

\begin{cons}\label{Cons of line bundle}
The quasi-coherent sheaf $\Sigma^{-2i}(\sO_{\mathfrak X})$ is a line bundle  for any even periodic formal spectral stack $\mathfrak X$ and all $i\in \mathbf Z$. By passing to $\pi_0$, we obtain a multiplicative family of line bundles $\pi_{2i}(\sO_{\mathfrak X})$
on the underlying classical stack $\mathfrak X^\heart$.
 \end{cons}

 On the other hand, the prismatization stacks also carry a canonical family of line bundles $\sO_{X^{\mathbbl \Delta}}\{i\}$ (and similarly for $X^{\mathcal N}$ and $X^\mathrm{Syn}$), called the \textit{Breuil-Kisin twists}. Since they are multiplicative, they are determined from the special case $i=-1$, and since they are sheaves, we it suffices to assume that $X=\Spf(R)$ for a qrsp ring $R$. By a special case of \cite[Lemma 9.1.4]{BLa} (see also \cite[Proposition 3.5]{Mondal}), the relevant Breuil-Kisin twist is in this case given as the second prismatic cohomology of the projective line
 $$
\mathbbl\Delta_R\{-1\}\,\simeq\, \mathrm H^2_{\mathbbl\Delta}(\mathbf P^1_R).
 $$

\begin{remark}
 The line bundles of Construction \ref{Cons of line bundle} may be described similarly.
Since the quasi-coherent sheaves $\pi_{2i}(\sO_{\mathfrak X})$ on $\mathfrak X^\heart$ are multiplicative and invertible, they satisfy $\pi_{2i}(\sO_{\mathfrak X})\simeq \pi_{-2}(\sO_{\mathfrak X})^{\otimes -i}$, and are as such already fully determined by the special case $i=-1$. The line bundle $\pi_{-2}(\sO_{\mathfrak X})$ is a sheaf on $\mathfrak X^\heart$, therefore it suffices by descent to  specify it in the special case of an even periodic formal affine $\mathfrak X = \Spf(A)$.
In that case, the corresponding $\pi_0(A)$-module may be viewed as the second cohomology of the complex projective line
$$
\pi_{-2}(A)\,\simeq \, A^2(\mathbf{CP}^1),
$$
where on the right, we have identified the spectrum $A$ with the cohomology theory that it represents.
\end{remark}

This formal resemblance suggests a relationship between the homotopy sheaves on the even periodic spectral stacks $\pounds X$, $\text{\textlira}X$, and $\$X$ on the one side, and the Breuil-Kisin twists on their underlying classical stacks $X^{\mathbbl\Delta},$ $X^\mathcal N$, and $X^\mathrm{Syn}$ on the other. As consequence of Theorem \ref{Main Theorem in Text}, we can make this precise.

\begin{prop}\label{Prop BK twists = homotopy groups}
For any $i\in \mathbf Z$, the $i$-th Breuil-Kisin twists on the prismatization stacks may be identified with the line bundles
$$\pi_{2i}(\mathcal O_{\begin{normalfont}\pounds\end{normalfont}X})\,\simeq \, \mathcal O_{X^{\mathbbl\Delta}}\{i\},
\quad
\pi_{2i}(\mathcal O_{\begin{normalfont}\text{\textlira}\end{normalfont}X})\,\simeq \,
\mathcal O_{X^{\mathcal N}}\{i\},
\quad
\pi_{2i}(\mathcal O_{\begin{normalfont}\$\end{normalfont}X})
\,\simeq \, \mathcal O_{X^{\mathrm{Syn}}}\{i\}.$$
\end{prop}

\begin{proof}
In light of the proof of Theorem \ref{Main Theorem in Text}, this follows  from the appearance of the Breuil-Kisin twists $\{i\}$ in the computations of the homotopy groups  $\pi_{2i}\left(\mathrm{TP}^{(-1)}\right)$ and $\pi_{2i}\left(\mathrm{TC}^{-(-1)}\right)$ from
Theorem \ref{Theorem Frobenius untwist for TP}, Proposition \ref{Prop homotopy groups of untiwsted TC^-}, and of the induced maps $\pi_{2i}(\varphi)$ and $\pi_{2i}(\mathrm{can})$ from
Proposition \ref{Prop Frob on homotopy}.
\end{proof}

This identification of the Breuil-Kisin twists admittedly comes as no surprise. It is already apparent from the Bhatt-Morrow-Scholze computation of \cite{BMS2}, recalled here as Theorem \ref{Thm BMS2 iso}. The same is less true for the identification between the two formal groups (always understood to be commutative, smooth, and 1-dimensional) on the prismatization stacks, established in Proposition \ref{Prop Drinfeld is Quillen} below, and extending the main result of \cite{Deven}.

\begin{cons}
By \cite[Corollary 4.3.8]{Even periodization}, the even periodization of the sphere spectrum $\Spec(\mathbf S)^\mathrm{evp}$ may be identified with the \textit{chromatic base spectrum} $\mM$. This spectral stack, studied in \cite{ChromaticCartoon} under the name $\mM{}^\mathrm{or}_\mathrm{FG}$, has the classical moduli stack of formal groups $\mM{}^\heart_\mathrm{FG}$ as its underlying classical stack. For any even periodic formal spectral stack $\mathfrak X$, there is therefore a canonical map
$$
\mathfrak X^\heart\,\simeq\, (\mathfrak X^\mathrm{evp})^\heart\to (\Spec(\mathbf S)^\mathrm{evp})^\heart\,\simeq \, \mM^\heart\,\simeq \, \mM{}^\heart_\mathrm{FG}.
$$
This map classifyies a canonical formal group $\w{\G}{}^{\mathcal Q_0}_{\mathfrak X}$ on the classical stack $\mathfrak X^\heart$, called the \textit{classical Quillen formal group of $\mathfrak X$}.
\end{cons}

In \cite{Drinfeld FG}, Drinfeld  introduced\footnote{Technically, Drinfeld only exhibited the formal group in \cite{Drinfeld FG}  over the stacks $X^{\mathbbl\Delta}$ and $X^{\mathcal N}$. But as explained in {\cite[Corollary 4.8]{Deven}} and {\cite[Remark 6.10]{Mathew-Mondal}}, the formal group descends along the \'etale cover $X^\mathcal N\to X^{\mathrm{Syn}}$ to a formal group over the syntomification as well.} a formal group over the prismatization stacks $X^{\mathbbl\Delta}$ (and similarly for $X^{\mathcal N}$ and $X^\mathrm{Syn}$), which we will denote by $\w{\mathbf G}{}^{\mathrm{Dr}}_{X^{\mathbbl\Delta}}$.
By unpacking \cite[Proposition 6.9]{Mathew-Mondal} and invoking quasi-syntomic descent, we find that this formal group may be defined by taking it in the case of a qrsp affine $X=\Spf(R)$ to be the classical formal  stack
$$
\w{\mathbf G}{}^{\mathrm{Dr}}_{X^\mathbbl\Delta} \, \simeq \, \Spf \left(\mathrm R\Gamma_{\mathbbl\Delta}(\mathrm B\mathbf G_{m, R})\right)\vert_{\CAlg_{\mathrm{ad}, R}^\heart}
$$
over $\Spf(R)^{\mathbbl\Delta} \simeq \Spf(\mathbbl\Delta_R)$.
Its strict abelian group structure is induced from that of the classifying stack $\mathrm B\mathbf G_m$, which may be viewed equivalently as coming from the strict abelian group structure on the multiplicative group $\mathbf G_m$, or as classifying the tensor product structure of line bundles.

\begin{remark}
The classical Quillen formal group of an even periodic formally affine stack $\mathfrak X = \Spf(A)$ may be written in a somewhat reminiscent form as
$$
\w{\G}{}^{\mathcal Q_0}_{\Spf(A)}\, \simeq \,\Spf(A^0(\mathrm{B}\mathbf C^\times)),
$$
where we once again identify  $A$ with the cohomology theory that it represents.
\end{remark}

\begin{remark}
The formal group structure on $\w{\G}{}^{\mathrm{Dr}}_{\Spf(R)^{\mathbbl\Delta}}$ encodes the addition law for the tensor product of algebraic line bundles on $R$-schemes under the first Chern class $c_1^{\mathbbl\Delta}$ of \cite[Section 7.5]{BLa}. Likewise, the classical Quillen formal group $\w{\G}{}^{\mathcal Q_0}_A$ for an even $\E$-ring $A$  encodes the addition law for the tensor product of topological line bundles under the first Chern class $c_1^A$, whose existence follows from the evenness (and therefore complex-orientability) of $A$.
\end{remark}

\begin{prop}\label{Prop Drinfeld is Quillen}
Let $X$ be a quasi-syntomic $p$-adic formal scheme $X$. There are canonical equivalences of formal groups
$$
\widehat{\mathbf G}{}^{\mathcal Q_0}_{\begin{normalfont}\pounds X\end{normalfont}}
\,\simeq\,
\widehat{\mathbf G}{}^{\mathrm{Dr}}_{X^{\mathbbl\Delta}},
\qquad
\widehat{\mathbf G}{}^{\mathcal Q_0}_{\begin{normalfont}\text{\textlira} X\end{normalfont}}
\,\simeq\,
\widehat{\mathbf G}{}^{\mathrm{Dr}}_{X^{\mathcal N}},
\qquad
\widehat{\mathbf G}{}^{\mathcal Q_0}_{\begin{normalfont}\$ X\end{normalfont}}
\,\simeq\,
\widehat{\mathbf G}{}^{\mathrm{Dr}}_{X^{\mathrm{Syn}}},
$$
over the respective classical stacks $X^{\mathbbl\Delta}$, $X^{\mathcal N},$ and $X^{\mathrm{Syn}}$.
\end{prop}

\begin{proof}
By quasi-syntomic descent, it suffices to restrict to the case when $X=\Spf(R)$ for a qrsp ring $R$. In that case, \cite[Theorem 4.7]{Deven} proves the desired result identifyng the classical Quillen formal group of $\pounds X$ and the Drinfeld formal group over $X^{\mathbbl\Delta}$.

On the other hand,  the Drinfeld formal group is defined by pullback along the canonical maps $X^?\to \Spf(\mathbf Z_p)^?$ for $?\in\{\mathbbl\Delta, \mathcal N, \mathrm{Syn}\}$, so it alternatively suffices to assume that $X=\Spf(\mathbf Z_p)$. By \cite[Corollary 4.8]{Deven}, there exists an essentially unique formal group $\widehat{\mathbf G}$ on $\Spf(\mathbf Z_p)^\mathrm{Syn}$ which satisfies both that
\begin{enumerate}[label =(\alph*)]
\item \label{a}The pullback of $\w{\G}$ to $\Spf(\mathbf Z_p)^{\mathbbl\Delta}$ recovers the Drinfled formal group.
\item \label{b}The dualizing line $\omega_{\widehat{\mathbf G}}$, i.e.\ the module of invariant differentials on the formal group $\w{\G}$, coincides with the Breuil-Kisin twist $\sO\{1\}$.
\end{enumerate}
The Drinfeld formal group on syntomification $\w{\G}{}^\mathrm{Dr}_{\mathrm{Spf}(\Z_p)^\mathrm{Syn}}$ certainly satisfies both of these conditions. It suffices to show that the classical Quillen formal group $\w{\G}{}^{\mathcal Q_0}_{\$\mathrm{Spf}(\Z_p)}$ does as well.

For \ref{a}, note that the map from an even periodic spectral stack to the chromatic base stack is essentially unique, so that the diagram of formal spectral stacks
$$
\begin{tikzcd}
\pounds \Spf(\Z_p) \arrow{dr}{}\arrow{rr}{} &  & \$\Spf(\Z_p)\arrow{dl}{}\\
& \mM &
\end{tikzcd}
$$
commutes. By passing to the underlying classical stacks, we obtain by Theorem \ref{Main Theorem in Text} a commuting diagram of classical stacks
$$
\begin{tikzcd}
 \Spf(\Z_p){}^{\mathbbl\Delta} \arrow{dr}{}\arrow{rr}{} &  & \Spf(\Z_p)^{\mathrm{Syn}}.\arrow{dl}{}\\
& \mM{}_\mathrm{FG}^\heart &
\end{tikzcd}
$$
Since the non-horizontal maps each classify their respective classical Quillen formal group, it follows that said formal groups are related by pullback along the horizontal map. Because we already know that $\widehat{\mathbf G}{}^{\mathcal Q_0}_{\begin{normalfont}\pounds \Spf(\Z_p)\end{normalfont}}
\simeq
\widehat{\mathbf G}{}^{\mathrm{Dr}}_{\Spf(\Z_p)^{\mathbbl\Delta}}$, it follows that the classical Quillen formal group of $\$\Spf(\Z_p)$ satisfies condition \ref{a}.

To verify \ref{b}, let us first observe that the dualizing line of a Quillen formal group is always given by 
$$
\omega_{\w{\G}{}^{\mathcal Q_0}_{\mathfrak X}}\,\simeq \, \pi_2(\sO_{\mathfrak X}).
$$
This is a classical comutation in the even periodic affine case, from which it may be deduced via flat descent, or it can also be shown directly by a computation in terms of the (non-classical, i.e.\ oriented) Quillen formal group \cite[Example 4.2.19]{Elliptic 2}. By specializing  to the case $\mathfrak X=\$\Spf(\Z_p)$, the identification between the Breuil-Kisin twists and homotopy sheaves of Proposition \ref{Prop BK twists = homotopy groups} then shows that the formal group $\widehat{\mathbf G}{}^{\mathcal Q_0}_{\begin{normalfont}\$ \Spf(\Z_p)\end{normalfont}}$ satisfies \ref{b}, and is therefore isomorphic to the Drinfeld formal group over the syntomification of $\Spf(\Z_p)$.

The remaining identification between the classical Drinfeld formal group of $\text{\textlira}X$ and the Drinfeld formal group over $X^\mathcal N$ follows from the already proven equivalence for $\$X$ and $X^\mathrm{Syn}$. That is because one the one hand, the map of even periodic spectral stacks $\text{\textlira}X\to \$X$ lives over $\mM$, and so by the same argument as before the classical Quillen fromal group from $\$X$ pulls back along the map of underlying classical stacks $X^{\mathcal N}\to X^\mathrm{Syn}$ to the classical Quillen formal group of $\text{\textlira X}$. On the other hand though, the Drinfeld formal group also pulls back along $X^{\mathcal N}\to X^\mathrm{Syn}$ by construction.
\end{proof}

\appendix
\section{Nygaard-decompletion via animation}

In Section \ref{Section 1}, we presented an account of how to \textit{Nygaard-decomplete} the variants of topological cyclic homology $\mathrm{TP}^\wedge_p$ and $(\mathrm{TC}^-)^\wedge_p$. We constructed these decompletions $\mathrm{TP}^{(-1)}$ and $\mathrm{TC}^{-(-1)}$ on the level of qrsp rings by localization before applying a Frobenius map. As such, we also view them as \textit{Frobenius-untwists}. They were then extended further by quasi-syntomic descent.
In this appendix, we  present a different approach to Nygaard decompleting $\mathrm{TP}^\wedge_p$ and $(\mathrm{TC}^-)^\wedge_p$, that we learned from Deven Manam. In place of localization and descent, this approach relies on \textit{animation}.

\begin{remark}
The idea of animation is a modern interpretation of Quillen's idea of the \textit{nonabelian derived category}. It was established in the $\i$-categorical setting by \cite[Section 5.5.8]{HTT}, and  dubbed its present name in 
\cite{Animation} by analogy that \textit{animation} endows a category with a homotopical \textit{soul} --  \textit{anima} in Latin. Formally, animation may be defined as
$\mathrm{Ani}(\mathcal C) :=\mathcal P_\Sigma(\mathcal C^\mathrm{cp})$, that is to say, it it formally adjoins all sifted colimits to the subcategory $\mathcal C^\mathrm{cp}\subseteq\mathcal C$ of the compact projective objects in $\mathcal C$.
\end{remark}

The idea is to obtain the Nygaard-decompletions $\mathrm{TP}^\mathrm{ani}$ and $\mathrm{TC}^{-\mathrm{ani}}$ by animating the functors $\mathrm{TP}^\wedge_p$ and $(\mathrm{TC}^-)^\wedge_p$. In order for this to produce the desired results, some care needs to be taken to specify enough structure to keep track of during animation. More precisely, we must keep track of the \textit{motivic filtrations} on the variants of topological cyclic homology, which had been first defined in \cite{BMS2}, and re-interpreted as special cases of the even filtration in \cite{HRW}:

\begin{theorem}[{\cite[Theorem 1.12]{BMS2}}]\label{Thm motivic filtrations exist}
Let $R$ be a quasi-syntomic $\Z_p$-algebra. There exist canonical complete filtered $\E$-ring 
$\mathrm{fil}_{\mathrm{mot}}^*(\mathrm{TP}(R)^\wedge_p)$ and $\mathrm{fil}_{\mathrm{mot}}^*(\mathrm{TC}^-(R)^\wedge_p)$, functorial in $R$, such that:
\begin{enumerate}[label =(\roman*)]
\item On the associated graded objects, they recover prismatic cohomology
$$
\mathrm{gr}_{\mathrm{mot}}^i(\mathrm{TP}(R)^\wedge_p)\,\simeq\,\widehat{\mathbbl\Delta}_R\{i\}[2i],
\quad
\mathrm{gr}_{\mathrm{mot}}^i(\mathrm{TC}^-(R)^\wedge_p)\,\simeq\,\mathrm{Fil}_{\mathcal N}^i\,\widehat{\mathbbl\Delta}_R\{i\}[2i].
$$
\item If $R$ is qrsp, they recover the Postnikov filtrations
$$
\mathrm{fil}_{\mathrm{mot}}^i(\mathrm{TP}(R)^\wedge_p)\,\simeq\,\tau_{\ge 2i}(\mathrm{TP}(R)^\wedge_p),
\quad
\mathrm{fil}_{\mathrm{mot}}^i(\mathrm{TC}^-(R)^\wedge_p)\,\simeq\,\tau_{\ge 2i}(\mathrm{TC}^-(R)^\wedge_p).
$$
\end{enumerate}
\end{theorem}

We may now define the Nygaard-decompletions $\mathrm{TP}^\mathrm{ani}$ and $\mathrm{TC}^{-\mathrm{ani}}$ by animating the motivic filtrations on $\mathrm{TP}^\wedge_p$ and $(\mathrm{TC}^-)^\wedge_p$.

\begin{cons}\label{Const TP, TC^- ani}
Let $\mathrm{Fil}\CAlg^\mathrm{cplt}_\mathrm{ad}$ denote the $\i$-category of filtered $\E$-rings for which the filtration is complete (see e.g.\ \cite[§1.6 Conventions, (4)]{HRW}), and  whose  underlying $\E$-ring is equipped with an adic topology, with respect to which each filtered level is adically complete. Since both adic completion and completion of a filtration and localizations, this $\i$-category has all small colimits, computed as the (both adic and filtration) completion of the corresponding colimit in $\mathrm{FilCAlg}$. 
The motivic filtrations, as recalled above in Theorem \ref{Thm motivic filtrations exist}, may be interpreted as  functors
 \begin{equation}\label{Fil mot eq}
\mathrm{fil}_{\mathrm{mot}}^*(\mathrm{TP})^\wedge_p,\, \mathrm{fil}_{\mathrm{mot}}^*(\mathrm{TC}^-)^\wedge_p :
\CAlg^{\mathrm{qsyn}}_{\Z_p}\to\mathrm{Fil}\CAlg^\mathrm{cplt}_\mathrm{ad}.
\end{equation}
Since $p$-completed polynomial $\Z_p$-algebras $\Z_p[t_1, \ldots, t_n]^\wedge_p$ are quasi-syntomic, which is to say that
$(\CAlg{}^{\mathrm{poly}}_{\Z_p}){}^\wedge_p\subseteq\CAlg_{\Z_p}^\mathrm{qsyn}$,
we can use the identification
$$
(\CAlg^\mathrm{ani}_{\Z_p})^\wedge_p\,\simeq \, \mathcal P_\Sigma\big((\CAlg{}^{\mathrm{poly}}_{\Z_p})^\wedge_p\big)
$$
of $p$-complete animated rings with the animation on the right-hand side. Due to the $\i$-category $\mathrm{FilCAlg}^\mathrm{cplt}_\mathrm{ad}$ being closed under colimits,
we find that the functors of \eqref{Fil mot eq} give rise to essentially unique sifted-colimit-preserving extensions
$$
\mathbf L\mathrm{fil}_{\mathrm{mot}}^*(\mathrm{TP})^\wedge_p,\, \mathbf L\mathrm{fil}_{\mathrm{mot}}^*(\mathrm{TC}^-)^\wedge_p :(\CAlg^{\mathrm{ani}}_{\Z_p})^\wedge_p\to\mathrm{Fil}\CAlg^\mathrm{cplt}_\mathrm{ad}.
$$
Composing with the filtration-forgetting functor $\mathrm{Fil}\CAlg^\mathrm{cplt}_\mathrm{ad}\to\CAlg_\mathrm{ad}$, we obtain the \textit{animated periodic topological cyclic homology} and \textit{animated negative topological cyclic homology} functors respectively
$$
(\mathrm{TP}^\mathrm{ani})^\wedge_p,\, (\mathrm{TC}^{-\mathrm{ani}})^\wedge_p : (\CAlg_{\Z_p}^\mathrm{ani})^\wedge_p \to \CAlg_\mathrm{ad}.
$$
By construction, there are canonical maps of adic $\E$-rings
\begin{equation}\label{comparisani}
\mathrm{TP}^\mathrm{ani}(R)^\wedge_p\to\mathrm{TP}(R)^\wedge_p, \qquad \mathrm{TC}^{-\mathrm{ani}}(R)^\wedge_p\to\mathrm{TC}^-(R)^\wedge_p,
\end{equation}
natural in the $p$-complete animated ring $R$, where in the codomain we implicitly apply the canonical functor $\CAlg^\mathrm{ani}_{\Z}\to\CAlg_{\Z}$ from animated rings to $\E$-algebras over $\Z$. On the other hand, the animated variants of topological cyclic homology are given for a $p$-complete animated ring $R$ explicitly by
$$
\mathrm{TP}^\mathrm{ani}(R)^\wedge_p \simeq \varinjlim_{i\to \infty}\mathbf L\mathrm{fil}_{\mathrm{mot}}^{-i}(\mathrm{TP}(R)^\wedge_p), \qquad
\mathrm{TC}^{-\mathrm{ani}}(R)^\wedge_p := \varinjlim_{i\to \infty}\mathbf L\mathrm{fil}_{\mathrm{mot}}^{-i}(\mathrm{TC}^-(R)^\wedge_p).
$$
It follows that the animated filtrations $\mathbf L\mathrm{fil}_{\mathrm{mot}}^*(\mathrm{TP})^\wedge_p,$ and $\mathbf L\mathrm{fil}_{\mathrm{mot}}^*(\mathrm{TC}^-)^\wedge_p$ may be interpreted as the respective motivic filtrations on $\mathrm{TP}^\mathrm{ani}(R)^\wedge_p$ and $\mathrm{TC}^{-\mathrm{ani}}(R)^\wedge_p$, i.e.
$$
\mathrm{fil}_\mathrm{mot}^*(\mathrm{TP}^\mathrm{ani}(R)^\wedge_p) :=\mathbf L\mathrm{fil}_\mathrm{mot}^*(\mathrm{TP}(R)^\wedge_p),
\quad
\mathrm{fil}_\mathrm{mot}^*(\mathrm{TC}^{-\mathrm{ani}}(R)^\wedge_p) :=\mathbf L\mathrm{fil}_\mathrm{mot}^*(\mathrm{TC}^-(R)^\wedge_p).
$$
\end{cons}

\begin{remark}
There is an obvious $p$-decompleted variant of Construction \ref{Const TP, TC^- ani}, where the only change is that we systematically drop  $p$-completeness. This results in functors
$$
\mathrm{TP}^\mathrm{ani},\, \mathrm{TC}^{-\mathrm{ani}} : \CAlg_{\Z}^\mathrm{ani} \to \CAlg_\mathrm{ad},
$$
as well as motivic filtrations on both. We still land in adic $\E$-rings, complete with respect to the augmentation ideal coming from the $\T$-action. It is easy to show that for a $p$-complete animated ring, $\mathrm{TP}^\mathrm{ani}(R)^\wedge_p$ and $\mathrm{TC}^{-\mathrm{ani}}(R)^\wedge_p$ of Construction \ref{Const TP, TC^- ani} coincide with the respective $p$-completions of $\mathrm{TP}^\mathrm{ani}(R)$ and $\mathrm{TC}^{-\mathrm{ani}}(R)$.
\end{remark}

To verify that what we have just animated into existence indeed constitutes a Nygaard-decompletion of $\mathrm{TP}^\wedge_p$ and $(\mathrm{TC}^-)^\wedge_p$, let us show directly that it coincides with the Frobenius untwists, the previously-constructed Nygaard-decompletion.

\begin{prop}
Under the usual inclusion $\CAlg_{\Z_p}^\mathrm{qsyn}\subseteq(\CAlg^\mathrm{ani}_{\Z_p})^\wedge_p$, there are canonical natural equivalences of adic $\E$-rings
$$
\mathrm{TP}^\mathrm{ani}(R)^\wedge_p\,\simeq \, \mathrm{TP}^{(-1)}(R), \qquad \mathrm{TC}^{-\mathrm{ani}}(R)^\wedge_p\,\simeq\,\mathrm{TC}^{-(-1)}(R)
$$
for all quasi-syntomic $\Z_p$-algebras $R$.
\end{prop}

\begin{proof}
Replace all the occurrences of $\mathrm{TP}^\wedge_p$ and $(\mathrm{TC}^-)^\wedge_p$
 in Construction \ref{Const TP, TC^- ani} with their Frobenius untwists $\mathrm{TP}^{(-1)}$, $\mathrm{TC}^{-(-1)}$ from Section \ref{Section 1}, on which the motivic filtration may for defined via the even filtration. In this way we obtain new functors
$$
\mathrm{fil}_\mathrm{mot}^*\big(\mathrm{TP}^{(-1),\, \mathrm{ani}}\big), \mathrm{fil}_\mathrm{mot}^*\big(\mathrm{TC}^{-(-1),\,\mathrm{ani}}\big) : (\CAlg_{\Z_p}^\mathrm{ani})^\wedge_p \to \mathrm{Fil}\CAlg^\mathrm{cplt}_\mathrm{ad},
$$
which fit for every $p$-complete animated ring $R$ into a canonical and natural square
$$
\begin{tikzcd}
 \mathrm{fil}_\mathrm{mot}^*\big(\mathrm{TC}^{-(-1),\,\mathrm{ani}}(R)\big) \arrow{d}{\mathrm{can}} \arrow{r}{} &  \mathrm{fil}_\mathrm{mot}^*(\mathrm{TC}^{-\mathrm{ani}}(R)) \arrow{d}{\mathrm{can}}\\
  \mathrm{fil}_\mathrm{mot}^*\big(\mathrm{TP}^{(-1),\,\mathrm{ani}}(R)\big) \arrow{r}{} &  \mathrm{fil}_\mathrm{mot}^*(\mathrm{TP}^{\mathrm{ani}}(R)).
\end{tikzcd}
$$
We claim that the horizontal maps are equivalences of complete filtered $\E$-rings. Since the animated versions are obtained by animation, the above square in question may be obtained as a sifted colimit of squares of the form
$$
\begin{tikzcd}
 \mathrm{fil}_\mathrm{mot}^*\big(\mathrm{TC}^{-(-1)}(S)\big) \arrow{d}{\mathrm{can}} \arrow{r}{} &  \mathrm{fil}_\mathrm{mot}^*(\mathrm{TC}^{-}(S)) \arrow{d}{\mathrm{can}}\\
  \mathrm{fil}_\mathrm{mot}^*\big(\mathrm{TP}^{(-1)}(S)\big) \arrow{r}{} &  \mathrm{fil}_\mathrm{mot}^*(\mathrm{TP}(S)).
\end{tikzcd}
$$
for $p$-completed polynomial rings $S\simeq \Z[t_1, \ldots, t_n]^\wedge_p$ with maps $S\to R$. Since all the filtered objects in sight are complete by definition, it suffices to show that the horizontal maps are equivalences in the corresponding associated graded diagrams for all $i\in\Z$
$$
\begin{tikzcd}
 \mathrm{gr}_\mathrm{mot}^i\big(\mathrm{TC}^{-(-1)}(S)\big) \arrow{d}{\mathrm{can}} \arrow{r}{} &  \mathrm{gr}_\mathrm{mot}^i(\mathrm{TC}^{-}(S)) \arrow{d}{\mathrm{can}}\\
  \mathrm{gr}_\mathrm{mot}^i\big(\mathrm{TP}^{(-1)}(S)\big) \arrow{r}{} &  \mathrm{gr}_\mathrm{mot}^i(\mathrm{TP}(S)).
\end{tikzcd}
$$
By Theorem \ref{Thm motivic filtrations exist}, and the corresponding computation of the motivic filtration of the non-Nygaard-complete versions for the Frobenius untwists that follows from the discussion in Section \ref{Section 1}, the associated graded diagram may be rewritten in terms of prisms as
$$
\begin{tikzcd}
 \mathrm{Fil}_{\mathcal N}^i\,{\mathbbl\Delta}_S\{i\}[2i] \arrow{d}{} \arrow{r}{} &  \mathrm{Fil}_{\mathcal N}^i\,\widehat{\mathbbl\Delta}_S\{i\}[2i] \arrow{d}{}\\
  \mathbbl\Delta_S\{i\}[2i]\arrow{r}{} &  \widehat{\mathbbl\Delta}_S\{i\}[2i].
\end{tikzcd}
$$
It remains to justify that the horizontal arrows are equivalences in the case of the completed polynomial ring $S=\Z_p[t_1, \ldots, t_n]^\wedge_p$. The equivalence statement for the upper horizontal arrow follows from the bottom one by passage to the Nygaard filtered pieces. On the other hand, the lower horizontal arrow being an equivalence is precisely the assertion that the prismatic cohomology of the completed polynomial ring is Nygaard-complete, and is proved in \cite[Proposition 5.8.2]{BLa}.
\end{proof}

\begin{remark}
One advantageous aspect of the approach to Nygaard-decompletion via animation is that it offers a natural suggestion for how to proceed outside the $p$-adic setting. That is to say, if $X$ is a scheme (classical or even derived) over $\Z$, the formal spectral stacks of the form  
$$
\pounds^\mathrm{gl}X\,:=\hspace{-0.3em} \varinjlim_{\substack{\Spec(A)\,\in \,\mathrm{SpAff}_{/X}\\ \mathrm{THH}(A)\,\in\,\CAlg^\mathrm{ev}}} \Spf(\mathrm{TP}^\mathrm{ani}(A)),
\quad\,
\text{\textlira}^\mathrm{gl}X\,:=\hspace{-0.3em} \varinjlim_{\substack{\Spec(A)\,\in \,\mathrm{SpAff}_{/X}\\ \mathrm{THH}(A)\,\in\,\CAlg^\mathrm{ev}}} \Spf(\mathrm{TC}^{-\mathrm{ani}}(A))^\mathrm{evp}
$$
naturally suggest themselves as the  \textit{global versions} of the  spherical Tate-loop and Nygaard-loop space of $X$ respectively. We are using ``global" here in analogy with \cite[Section 6.4]{BLa}, since these spectral stacks will  recover (a Nygaard-decompleted version of) global prismatic cohomology in the sense of \cite[Construction 6.4.6]{BLa} on $\E$-rings of global sections. Their underlying classical stacks $(\pounds^\mathrm{gl}X)^\heart$ and $(\text{\textlira}^\mathrm{gl}X)^\heart$ may be taken to be versions of ``prismatization and filtered prismatization of $X$ over $\Z$", at least under some assumptions on $X$ (e.g.\ quasi-syntomic over $\Z$). The author is presently unaware just how computationally feasible this suggestion is; we leave such considerations to future work.
\end{remark}

\end{document}